\documentclass[final,a4paper]{scrartcl}
\pdfoutput=1
\KOMAoptions{parskip=half}
\KOMAoptions{twoside=semi}
\title{\vspace*{-4ex}Ergodic Properties of \\ Quantum Birth and Death Chains}
\author{
	David~B\"ucher\thanks{E-mail address: david.buecher@math.uni-hamburg.de}~$^,$\thanks{E-mail addresses: gaertner@mathematik.tu-darmstadt.de, kuemmerer@mathematik.tu-darmstadt.de, reusswig@mathematik.tu-darmstadt.de}\;,~ Andreas~G\"artner$^{\dagger}$,~ Burkhard~K\"ummerer$^{\dagger}$, \\ Walter~Reu\ss wig$^{\dagger}$,~ Kay~Schwieger\thanks{E-mail address: kay.schwieger@helsinki.fi}~$^{,\dagger}$,~ Nadiem~Sissouno$^{\dagger}$\\[1.4ex]%
	\begin{small}$^{\ast}$ Fachbereich Mathematik, Universit\"{a}t Hamburg,\end{small}\\[-1.2ex]%
	\begin{small}Bundesstr. 55, 20146 Hamburg, Germany\end{small}\\%
	\begin{small}$^{\dagger}$ Fachbereich Mathematik, Technische Universit\"{a}t Darmstadt,\end{small}\\[-1.2ex]%
	\begin{small}Schlo{\ss}gartenstr.~7, 64289 Darmstadt, Germany\end{small}\\%
	\begin{small}$^{\ddagger}$ Matematiikan ja tilastotieteen laitos, Helsingin yliopisto,\end{small}\\[-1.2ex]%
	\begin{small}Gustaf H\"allstr\"omin katu 2b, 00014 Helsinki, Finland\end{small}\\%
}
\date{June 17, 2013}

\usepackage{fancyhdr} 
\usepackage[utf8x]{inputenc}
\usepackage[T1]{fontenc}
\usepackage{lmodern}
\usepackage{fixltx2e}
\usepackage{microtype}
\usepackage[english]{babel}
\usepackage{amsmath}  
\usepackage{amssymb, nicefrac, mathtools, bbm}
\usepackage{xspace}
\usepackage{varioref}
	\labelformat{equation}{(#1)}
	\renewcommand{\eqref}{\ref}
\usepackage{enumitem}
\usepackage[arrow, matrix, curve]{xy}
	\CompileMatrices
\usepackage{tikz} 
	\usetikzlibrary{decorations.pathmorphing} 
	\usetikzlibrary{arrows,patterns} 
\usepackage{mathrsfs} 
\usepackage[plainpages=false, final=true, pdfborder={0 0 0.1}]{hyperref} 

\usepackage{amsthm} 

	\newtheoremstyle{plain_big_preskip}%
	{1.25\baselineskip}
	{.5\baselineskip}
	{\itshape}
	{}
	{\bfseries}
	{.}
	{.5em}
	{}

	\theoremstyle{plain_big_preskip}
		\swapnumbers
		\newtheorem{thm}{Theorem}[section]
		
		\newtheorem{prop}[thm]{Proposition}
	\theoremstyle{definition}
		\newtheorem{defn}[thm]{Definition}
	\theoremstyle{remark}
		\newtheorem{rmk}[thm]{Remark}
		\newtheorem{expl}[thm]{Example}
		
	\numberwithin{equation}{section}
	
	\newcounter{secfigure}[section]
	
	\makeatletter
		\let\oldfigure\figure
		\let\endoldfigure\endfigure
		\renewenvironment{figure}{\stepcounter{secfigure}\oldfigure}{\endoldfigure}
	\makeatother
		
\newcommand*{\alg}{\mathcal}
\newcommand*{\hilb}{\mathscr}
\newcommand*{\N}{\mathbb N}		
\newcommand*{\R}{\mathbb R}		
\newcommand*{\C}{\mathbb C}		
\newcommand*{\one}{\mathbbm 1}		
\renewcommand*{\H}{\hilb H}		

\newcommand*{\B}{\hilb B}		
\newcommand*{\BH}{\B(\H)}		
\newcommand*{\D}{\mathbb D}		
\newcommand*{\tensor}{\otimes}		
\newcommand{\fix}[1]{\alg F(#1)}

\DeclarePairedDelimiter{\abs}{\lvert}{\rvert}		
\DeclarePairedDelimiter{\norm}{\lVert}{\rVert}		
\DeclarePairedDelimiter{\scal}{\langle}{\rangle}	
\newcommand{\bra}[1]{\langle #1 \vert} 
\newcommand{\ket}[1]{\vert #1 \rangle} 

\DeclareMathOperator{\UCP}{UCP}		
\DeclareMathOperator{\Tr}{Tr}		
\DeclareMathOperator{\diag}{diag}	
\DeclareMathOperator{\supp}{supp}	
\DeclareMathOperator{\lin}{span}	

\newcommand{\ie}{\mbox{i.\,e.}\xspace}		
\newcommand{\eg}{\mbox{e.\,g.}\xspace}		
\newcommand{\cf}{\mbox{cf.}\xspace}		
\newcommand*{\etal}{\mbox{et~al.}\xspace}

\begin{document}
\maketitle
\sloppy

\begin{abstract}\vspace*{-4ex}
	We study a class of quantum Markov processes that, on the one hand, is inspired by the micromaser experiment in quantum optics and, on the other hand, by classical birth and death processes. We prove some general geometric properties and irreducibility for non-degenerated parameters. Furthermore, we analyze ergodic properties of the corresponding transition operators. For homogeneous birth and death rates we show how these can be fully determined by explicit calculation. As for classical birth and death chains we obtain a rich yet simple class of quantum Markov chains on an infinite space, which allow only local transitions while having divers ergodic properties.
\end{abstract}

\section{Introduction}	\label{sec:intro}
In the present paper we examine a class of quantum Markov chains on an infinite dimensional space that allows elementary computations, but also offers a great diversity of processes of varying ergodic properties.

The construction of our Markov chains is adopted from a finite dimensional variant introduced by R.\,Gohm, B.\,K\"ummerer, and T.\,Lang \cite{GKL2006}. Their work was motivated by the micromaser experiment in quantum optics, where a photon of a single-mode electromagnetic field, \ie a quantum harmonic oscillator, can be created (``birth'') or annihilated (``death'') by the interaction with a two-level atom (see, \eg, \cite{VBW+2000} and \cite{WBK+2000}). This effect was also studied by S.\,Gleyzes \etal \cite{GKG+2007} who were able to detect the life-time of a photon in a concrete experimental design. Some ergodic properties of the repeated interaction were investigated by L.\,Bruneau and C.-A.\,Pillet \cite{BP2009}. Their results can partially be recovered from our more general approach. A continuous-time generalization of the interaction including dissipation was also investigated by F.\,Fagnola, R.\,Rebolledo, and C.\,Saavedra \cite{FRS1994} (see also \cite{CFL2000}), who obtain results similar to those of the present paper.

According to the Jaynes-Cummings model of the micromaser experiment the interaction with a single atom shifts the energy level of the field only to neighboring levels. The Markov chains studied here may be interpreted as generalized Jaynes-Cummings interactions. As in the physical model they only allow local transitions. Mathematically, each of our processes extends a classical birth and death chain to a non-commutative framework (\cf Remark \ref{rm:whyBnD}).

Although a general framework of birth and death chains is still missing in quantum probability theory, we offer a wide class of quantum Markov chains that should be covered by such a framework in our opinion. We invite the reader to use these chains as a prototype for examples or---maybe more important---counterexamples for typical phenomena of quantum Markov chains on infinite spaces.

This paper is organized as follows: The studied class of quantum Markov chains is introduced in Section \ref{sec:bndchain} by constructing the corresponding transition operators $T_\psi$. We also characterize under which conditions this transition operator is an extreme point in the set of all unital completely positive normal maps. Section \ref{sec:ergProp} is devoted to the study of ergodic properties of our quantum birth and death chains in general. In particular, we show that for non-extremal parameters the transition operator is irreducible and establish criteria that ensure weak mixing.
In Section \ref{sec:toy} we restrict ourselves to a subclass of transition operators that generalizes homogeneous birth and death chains. For these toy examples we examine their mixing properties and the existence of invariant normal states. Finally, Section \ref{sec:summary} contains a summary and visualizations of our results ordered by the examined properties.

\section{Notation and Preliminaries}\label{sec:prelim}
Throughout this paper $\N$ denotes the set of natural numbers including zero and \mbox{$\mathbb D := \{ \zeta \in \C: \abs \zeta  \le 1\}$} denotes the closed unit disc in the complex plane.
By $\H$ we refer to a Hilbert space with scalar product $\scal{\,\cdot\,,\,\cdot\,}$ linear in the first component. We write $\B(\H)$ for the $^*$\nobreakdash-algebra of all bounded linear operators on $\H$. The identity operator on $\H$ is denoted by $\one$. For an (orthogonal) projection $p \in \BH$ we denote by $p^\bot := \one - p$ its complement. 
A~\emph{state} on $\B(\H)$ is a bounded linear functional $\varphi:\B(\H) \to \C$ satisfying $\varphi(\one) = 1$ and $\varphi(x^*x) \ge 0$ for every $x \in \B(\H)$. It is called \emph{faithful} if $\varphi(x^*x) = 0$ implies $x=0$. The state is called \emph{normal} if there is a trace class operator $\rho \in \BH$ with $\varphi(x) = \Tr(\rho \, x)$ for every $x \in \BH$.
Each unit vector $\xi \in \H$ gives rise to a normal state by putting $\varphi_\xi(x) := \langle x\,\xi,\xi \rangle$, $x \in \BH$, called the \emph{vector state} of $\xi$. The vector states, also called \emph{pure states}, are exactly the extreme points of the convex set of all normal states on $\B(\H)$.

In this paper we almost exclusively consider the Hilbert space $\H := \ell^2(\N)$ of all square summable complex sequences equipped with the canonical orthonormal basis $(e_n)_{n \in \N}$. For $m,n \in \N$ we denote by $e_{m,n}:\H\to\H$, $\xi\mapsto\scal{\xi,e_n}e_m$ the canonical matrix unit. For brevity we set $p_n:=e_{n,n}$ and $p_{[m,n]}:=\sum_{k=m}^n p_k$. An operator $x\in\BH$ can also be represented by the infinite matrix $(x_{m,n})_{m,n \in \N}$ of its \emph{coefficients} $x_{m,n}:=\scal{x\, e_n,e_m}$. The operator $x$ is called \emph{diagonal} if the matrix $(x_{m,n})_{m,n \in \N}$ is diagonal, \ie, if $x_{m,n}=0$ for all~$m\neq n$. We denote by $\ell^\infty(\N) \subseteq \BH$ the abelian subalgebra of all diagonal operators. Notice that $\ell^\infty(\N)$ is the von Neumann algebra generated by the projections $p_n$, $n\in\N$. A~normal state $\varphi$ on $\BH$ is called \emph{diagonal} if the corresponding trace class operator is diagonal, or equivalently, if $\varphi(e_{m,n})=0$ for all $m\neq n$.

The linear operators on the Hilbert space $\C^n$ with orthonormal basis $e_1,\ldots, e_n$ can be identified with the algebra $M_n$ of all complex $(n \times n)$-matrices. Then an operator of the tensor product algebra $\B(\H \tensor \C^n)$ can be written as an $(n \times n)$-matrix with entries in $\B(\H)$. Note that  $\B(\H \tensor \C^n)$ is linearly spanned by all operators
\begin{equation*}
	x \tensor y = \begin{pmatrix}
		y_{1,1} \, x & \dots & y_{1,n} \, x \\
		\vdots && \vdots \\
		y_{n,1} \, x & \dots & y_{n,n} \, x
	\end{pmatrix}
\end{equation*}
with $x \in \B(\H)$ and $y = (y_{i,j})_{1 \le i,j \le n} \in M_n$. 

Let $\H$ be a Hilbert space and let $T:\B(\H) \to \B(\H)$ be a linear map. Then $T$ is called \emph{unital} if $T(\one) = \one$. 
The map $T$ is called \emph{positive} if for every positive operator $x \in \B(\H)$ also $T(x)$ is a positive operator. It is called \emph{$n$\nobreakdash-positive} if the map
\begin{equation*}
	T_n:\B(\H \tensor \C^n) \to \B(\H \tensor \C^n), \quad (x_{i,j})_{i,j} \mapsto ( T(x_{i,j}))_{i,j}
\end{equation*}
is positive. If $T$ is $n$\nobreakdash-positive for all $n \in \N$ then it is called \emph{completely positive}. It follows that a completely positive map is continuous with respect to the operator norm on $\B(\H)$ with $\norm T = \norm{T(\one)}$.

A positive linear map $T:\B(\H) \to \B(\H)$ is called \emph{normal} if $T(\sup_i x_i) = \sup_i T(x_i)$ for every bounded increasing family $(x_i)_{i \in I}$ of positive operators $x_i \in \B(\H)$. A unital completely positive normal linear map $T:\B(\H) \to \B(\H)$ is briefly called a \emph{ucp\nobreakdash-map}. The convex set of all such maps is denoted by $\UCP(\H)$.

Let $T:\B(\H) \to \B(\H)$ be a ucp\nobreakdash-map. Then by a theorem of K.~Kraus \cite{Kraus1971} there is a family $(a_i)_{i \in I}$ of operators $a_i \in \BH$ such that
\begin{equation*}
	T(x) = \sum_{i \in I} a_i^* x a_i 
\end{equation*}
for every $x \in \BH$, where the sum converges in the strong operator topology. This decomposition of $T$ is known as \emph{Kraus decomposition},  the operators $a_i$ are called \emph{Kraus operators} of $T$. In this paper we will only deal with \emph{finite} Kraus decompositions, \ie, the index set $I$ is finite. In this case the ucp\nobreakdash-map $T$ is an \emph{extreme point} in the convex set $\UCP(\H)$ if and only if it admits a Kraus decomposition such that the operators $a_j^* a_i \in \B(\H)$ with $i,j \in I$ are linearly independent (\cf \cite{Choi1975}).

\section*{The State Space of \texorpdfstring{{\boldmath $M_2$}}{M2}}

For each state $\psi$ on $M_2$ there is a unique positive operator $\rho \in M_2$ with $\Tr(\rho)=1$ such that $\psi(x) = \Tr(\rho \, x)$. This operator can be written in the form
\begin{equation*}
	\rho = \tfrac12 \begin{pmatrix}
		1+z & x + iy \\
		x-iy & 1 - z
	\end{pmatrix} 
\end{equation*}
where $x,y,z \in \R$ with $x^2 + y^2 + z^2 \le 1$. By means of this parametrization the convex set of states on $M_2$ is affinely isomorphic to the 3-dimensional Euclidean unit ball (see Figure~\ref{fig:bloch}). For $z=1$ (north pole) and $z=-1$ (south pole) we denote the corresponding states by $\psi_+$ and $\psi_-$, respectively.

To simplify our computations we set $\lambda := \tfrac12(1+z)$. Then the off-diagonal entry of $\rho$ is a complex number of absolute value smaller than $\sqrt{\lambda(1-\lambda)}$. Hence $\rho$ is of the form
\begin{equation}
	\label{eq:density}
	\rho = \begin{pmatrix}
		\lambda & \bar\zeta \sqrt{\lambda(1-\lambda)} \\
		\zeta \sqrt{\lambda(1-\lambda)}  & 1-\lambda
	\end{pmatrix}
\end{equation}
for some $0 \le \lambda \le 1$ and some $\zeta \in \D = \{\theta \in \C : \abs \theta \le 1\}$. For the parameters $\lambda =0$ and $\lambda=1$ we agree to set $\zeta := 0$. This convention will allow us to phrase our theorems more consistently. 

We will frequently restrict our attention to two subsets of parameters. One typical choice will be a diagonal state, \ie $\zeta = 0$. In the physical literature this is sometimes called a temperature state. The other typical choice will be a pure state, \ie $\lambda \in \{0,1\}$ or $\abs\zeta =1$.

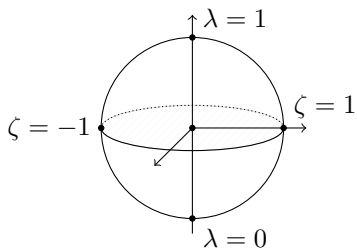
\begin{figure}[bht]
	\centering
	\begin{tikzpicture}
		\draw  (0,0)  circle  (1.2cm);
		\draw[pattern=north east lines,opacity=0.15] (0,0) ellipse (1.2cm and 0.3cm);
		\draw  (-1.2,0)  arc (180:360:1.2cm  and  0.3cm);
		\draw[densely dotted]  (1.2,0)  arc (0:180:1.2cm  and  0.3cm);

		\draw[->] (0,0) -- (1.5,0);
		\draw[->] (0,-1.4) -- (0,1.5);
		\draw[->] (0,0) -- (-0.5,-0.5);

		\filldraw[black] (0,0) circle (1pt);
		\filldraw[black] (0,1.2) node[above right]{\small $\lambda=1$} circle (1pt);
		\filldraw[black] (0,-1.2) node[below right]{\small $\lambda=0$} circle (1pt);
		\filldraw[black] (-1.2,0) node[left]{\small $\zeta=-1$} circle (1pt);
		\filldraw[black] (1.2,0) node[above right]{\small $\zeta=1$} circle (1pt);
	\end{tikzpicture}
	\caption{Bloch\;\textquoteleft sphere\textquoteright\ in our parametrization.}
	\label{fig:bloch}
\end{figure}

\section{A Class of Quantum Birth and Death Chains}	\label{sec:bndchain}
In this section we describe the class of ucp\nobreakdash-maps that we will study throughout the paper. Each of these map is the transition operator of a quantum Markov process in the sense of \cite{AFL1982}. For a recent overview on quantum Markov processes in general we recommend \cite{Kuemmerer2006}.

Fix two sequences $(\alpha_n)_{n \in \N}$ and $(\beta_n)_{n \in \N}$ of numbers $-1 \le \alpha_n, \beta_n \le 1$ satisfying \mbox{$\alpha_n^2 + \beta_n^2 = 1$} for every $n \in \N$, $\alpha_0 = 1$, and $\beta_n \neq 0$ for every $n \ge 1$. We refer to $\alpha_n$ and $\beta_n$ as the \emph{model parameters}. 
Consider the Hilbert space $\H := \ell^2(\N)$ equipped with the canonical orthonormal basis $(e_n)_{n \in \N}$. Denote by $a,b \in \B(\H)$ the diagonal operators with $a \, e_n := \alpha_n \, e_n$ and $b \, e_n := \beta_n \, e_n$ for every $n \in \N$, \ie
\begin{align*}
	a &= \begin{pmatrix}
		1 &  \\
		 & \alpha_1   \\
		&  & \alpha_2   \\
		&&  & \ddots  
	\end{pmatrix} \;,
	&
	b &= \begin{pmatrix}
		0 &  \\
		 & \beta_1   \\
		&  & \beta_2   \\
		&&  & \ddots  
	\end{pmatrix} \;.
\end{align*}
Write $s \in \B(\H)$ for the isometric shift of the basis, \ie $s \, e_n = e_{n+1}$ for every $n \in \N$. Moreover, for each state $\psi$ on $M_2$ denote by $P_\psi: \B(\H \tensor \C^2) \to \B(\H)$ the linear extension of $P_\psi(x \tensor y) := \psi(y) \cdot x$ for all $x\in \BH$, $y \in M_2$.

\pagebreak[3]
We define a unitary operator $u \in \B(\H \tensor \C^2)$ by
\begin{equation}	\label{eq:unitary}
	u := \begin{pmatrix}
		s^* a s & i s^* b \\
		i bs & a
	\end{pmatrix} \;.
\end{equation}

\begin{defn}
	For a state $\psi$ on $M_2$ the unital completely positive normal map
	\begin{equation*}
		T_\psi:\B(\H) \to \B(\H), \quad T_\psi(x) := P_\psi \bigl( u^*(x \tensor \one) u \bigr) 
	\end{equation*}
	is called the \emph{transition operator} associated with the state $\psi$.
\end{defn}
The map $T_\psi$ will be the focal point of our studies in this paper. For convenience we included a diagram of the action of $T_\psi$ on the canonical matrix units of $\B(\H)$ in Figure~\ref{fig:action}. The necessary computations can be found in Appendix \ref{appendix}.

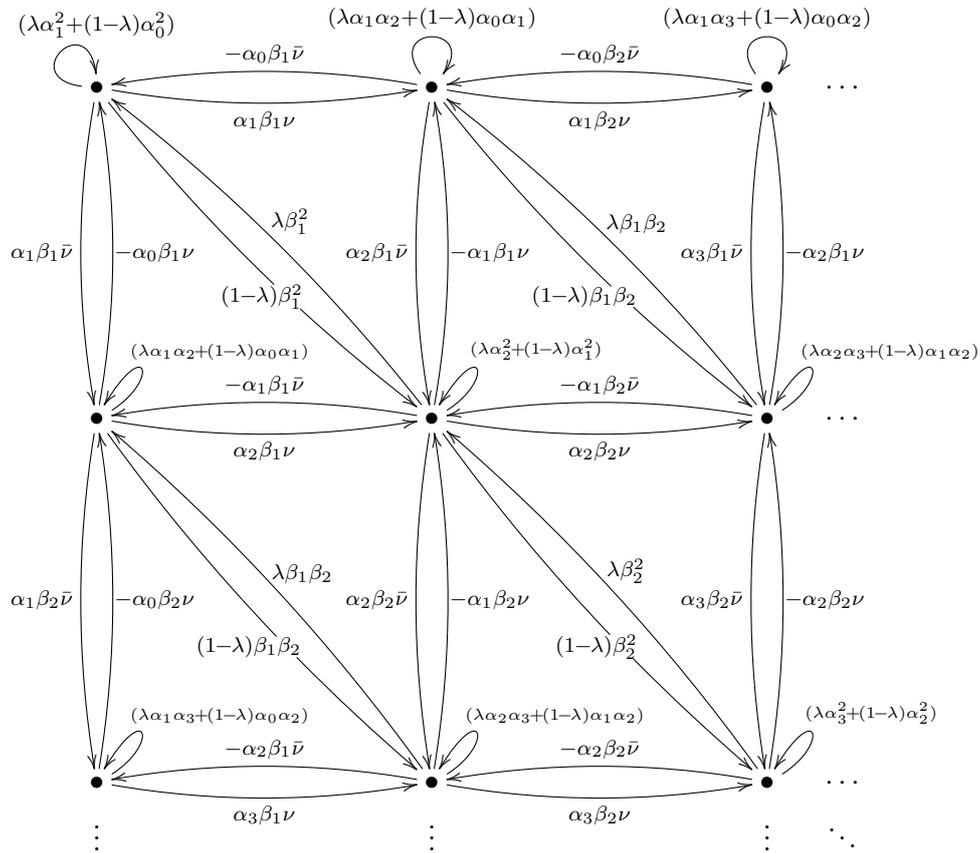
\begin{figure}[bht]
	\[
	\begin{xy}
		\xymatrix{
			\bullet
			\ar@/_/[rrrrdddd]|(.6){(1\hspace{-.1ex}-\hspace{-.1ex}\lambda)\beta_1^2\hspace{4ex}}
			\ar@/_/[rrrr]_{\alpha_1\beta_1\nu}
			\ar@/_/[dddd]_{\alpha_1\beta_1\bar\nu}
			\ar@(l,u)^(.7){(\lambda\alpha_1^2+(1\hspace{-.1ex}-\hspace{-.1ex}\lambda)\alpha_0^2)}
			& & & & \bullet
			\ar@/_/[rrrrdddd]|(.6){(1\hspace{-.1ex}-\hspace{-.1ex}\lambda)\beta_1\beta_2\hspace{6ex}}
			\ar@/_/[llll]_{-\alpha_0\beta_1\bar\nu}
			\ar@/_/[rrrr]_{\alpha_1\beta_2\nu}
			\ar@/_/[dddd]_{\alpha_2\beta_1\bar\nu}
			\ar@(lu,ru)^{(\lambda\alpha_1\alpha_2+(1\hspace{-.1ex}-\hspace{-.1ex}\lambda)\alpha_0\alpha_1)}
			& & & & \bullet
			\ar@/_/[llll]_{-\alpha_0\beta_2\bar\nu}
			\ar@/_/[dddd]_{\alpha_3\beta_1\bar\nu}
			\ar@(lu,ru)^{(\lambda\alpha_1\alpha_3+(1\hspace{-.1ex}-\hspace{-.1ex}\lambda)\alpha_0\alpha_2)}
			& \hspace{-1em}\cdots
			\\ \\ \\ \\
			\bullet
			\ar@/_/[rrrrdddd]|(.6){(1\hspace{-.1ex}-\hspace{-.1ex}\lambda)\beta_1\beta_2\hspace{6ex}}
			\ar@/_/[rrrr]_{\alpha_2\beta_1\nu}
			\ar@/_/[uuuu]_{\hspace{-0.5ex}-\alpha_0\beta_1\nu}
			\ar@/_/[dddd]_{\alpha_1\beta_2\bar\nu}
			\ar@`{+(9,4),+(-2,9)}_(.56){\text{\begin{tiny}$(\lambda\alpha_1\alpha_2\!\!+\!\!(1\!\!-\!\!\lambda)\alpha_0\alpha_1)$\end{tiny}}}
			& & & & \bullet
			\ar@/_/[rrrrdddd]|(.6){(1\hspace{-.1ex}-\hspace{-.1ex}\lambda)\beta_2^2\hspace{4ex}}
			\ar@/_/[uuuullll]_{\hspace{-1ex}\lambda\beta_1^2}
			\ar@/_/[llll]_{-\alpha_1\beta_1\bar\nu}
			\ar@/_/[rrrr]_{\alpha_2\beta_2\nu}
			\ar@/_/[uuuu]_{\hspace{-0.5ex}-\alpha_1\beta_1\nu}
			\ar@/_/[dddd]_{\alpha_2\beta_2\bar\nu}
			\ar@`{+(9,4),+(-2,9)}_(.56){\text{\begin{tiny}$(\lambda\alpha_2^2\!\!+\!\!(1\!\!-\!\!\lambda)\alpha_1^2)$\end{tiny}}}
			& & & & \bullet
			\ar@/_/[uuuullll]_{\hspace{-1ex}\lambda\beta_1\beta_2}
			\ar@/_/[llll]_{-\alpha_1\beta_2\bar\nu}
			\ar@/_/[uuuu]_{\hspace{-0.5ex}-\alpha_2\beta_1\nu}
			\ar@/_/[dddd]_{\alpha_3\beta_2\bar\nu}
			\ar@`{+(9,4),+(-2,9)}_(.56){\text{\begin{tiny}$(\lambda\alpha_2\alpha_3\!\!+\!\!(1\!\!-\!\!\lambda)\alpha_1\alpha_2)$\end{tiny}}}
			& \hspace{-1em}\cdots
			\\ \\ \\ \\
			**[d] \underset{\underset{\vdots}{ }}{\bullet} 
			\ar@/_/[rrrr]_{\alpha_3\beta_1\nu}
			\ar@/_/[uuuu]_{\hspace{-0.5ex}-\alpha_0\beta_2\nu}
			\ar@`{+(9,4),+(-2,9)}_(.56){\text{\begin{tiny}$(\lambda\alpha_1\alpha_3\!\!+\!\!(1\!\!-\!\!\lambda)\alpha_0\alpha_2)$\end{tiny}}}
			& & & & **[d] \underset{\underset{\vdots}{ }}{\bullet}  
			\ar@/_/[uuuullll]_{\hspace{-1ex}\lambda\beta_1\beta_2}
			\ar@/_/[llll]_{-\alpha_2\beta_1\bar\nu}
			\ar@/_/[rrrr]_{\alpha_3\beta_2\nu}
			\ar@/_/[uuuu]_{\hspace{-0.5ex}-\alpha_1\beta_2\nu}
			\ar@`{+(9,4),+(-2,9)}_(.56){\text{\begin{tiny}$(\lambda\alpha_2\alpha_3\!\!+\!\!(1\!\!-\!\!\lambda)\alpha_1\alpha_2)$\end{tiny}}}
			& & & & **[d] \underset{\underset{\vdots}{ }}{\bullet}  
			\ar@/_/[uuuullll]_{\hspace{-1ex}\lambda\beta_2^2}
			\ar@/_/[llll]_{-\alpha_2\beta_2\bar\nu}
			\ar@/_/[uuuu]_{\hspace{-0.5ex}-\alpha_2\beta_2\nu}
			\ar@`{+(9,4),+(-2,9)}_(.56){\text{\begin{tiny}$(\lambda\alpha_3^2\!\!+\!\!(1\!\!-\!\!\lambda)\alpha_2^2)$\end{tiny}}}
			& **[d] \hspace{-1em}\underset{\underset{\ddots}{ }}{\cdots}  
		}
	\end{xy}
	\]
	\caption{Action of $T_\psi$ where $i\zeta\sqrt{\lambda(1-\lambda)}$ is abbreviated by $\nu$.}
	\label{fig:action}
\end{figure}

\pagebreak
\begin{rmk}	\label{rm:whyBnD}
	The map $T_\psi$ can be regarded as the transition operator of a \emph{``quantum birth and death chain''} for several reasons:\begin{enumerate}
	\item
		Let $\psi$ be a state on $M_2$ with parameters $0 \le \lambda \le 1$ and $\zeta \in \D$. In the special case $\zeta = 0$ the abelian subalgebra $\ell^\infty(\N) \subseteq \B(\H)$ of all diagonal operators is invariant for $T_\psi$ in the following strong sense: If $x \in \BH$ is a diagonal operator then also $T_\psi(x)$ is a diagonal operator and if $\varphi$ is a diagonal normal state on $\BH$ then also $\varphi \circ T_\psi$ is a diagonal state. The restriction of $T_\psi$ to the subalgebra $\ell^\infty(\N)$ is the transition operator of the classical birth and death chain whose transition graph on the state space $\N$ is shown in Figure~\ref{fig:classical_bnd_chain}.
 
		\begin{figure}[thb]
			\centering
			~\begin{xy} \xymatrix@C=1.8cm {
				{\text{\textcircled{\raisebox{-0.9pt}{\small 0}}}}
				\ar@(ld,lu)@<-2pt>^{(1-\lambda) \alpha_0^2 + \lambda \alpha_1^2}
				\ar@/^/[r]^{\lambda\beta_1^2}
				&
				{\text{\textcircled{\raisebox{-0.9pt}{\small 1}}}}
				\ar@/^/[l]^{(1-\lambda) \beta_1^2}
				\ar@(ul,ur)^{(1-\lambda)\alpha_1^2 + \lambda \alpha_2^2}
				\ar@/^/[r]^{\lambda \beta_2^2}
				&
				{\text{\textcircled{\raisebox{-0.9pt}{\small 2}}}}
				\ar@/^/[l]^{(1-\lambda) \beta_2^2}
				\ar@(ul,ru)^{(1-\lambda) \alpha_2^2 + \lambda \alpha_3^2}
				\ar@/^/[r]^{\lambda \beta_3^2}
				&
				\cdots \ar@/^/[l]^{(1-\lambda) \beta_3^2}
			} \end{xy}
			\caption{Classical birth and death chain in the Schr\"odinger picture.}
			\label{fig:classical_bnd_chain}
		\end{figure}

		For an arbitrary $\zeta\in\D$ this classical transition operator is the composition \mbox{$P_{\ell^\infty(\N)} \circ T_\psi:\ell^\infty(\N) \to \ell^\infty(\N)$}, where  $P_{\ell^\infty(\N)}$ denotes the conditional expectation onto $\ell^\infty(\N)$ given by
		\begin{equation*}
			P_{\ell^\infty(\N)}(x) := \sum_{n=0}^\infty p_n x p_n = \begin{pmatrix}
				x_{0,0} & 0 \\
				0 & x_{1,1} & 0 \\
				& \ddots & \ddots & \ddots 
			\end{pmatrix} \;.
		\end{equation*}
	\item 
		Let $\psi$ be a state on $M_2$ with arbitrary parameters $0 \le \lambda \le 1$ and $\zeta \in \D$ as in Equation \eqref{eq:density}.
		Then for every $m \le n$ we have (\cf Proposition \vref{prop:neighboringTransitions})
		\begin{equation*}
			p_{[m+1,n-1]} \le T_\psi(p_{[m,n]}) \le p_{[m-1, n+1]} \;.
		\end{equation*}
		In this sense $T_\psi$ admits only local transitions to neighboring subspaces. (This is one of the abstract characterizations of quantum birth and death chains that F.\,Haag proposed in his master thesis \cite{Haag2002}). The image of $p_{[m,n]}$ can be visualized as follows:
		\[
			T_\psi(p_{[m,n]}) \;= \mbox{\begin{small} $\;\left(
				\begin{array}{c|c|c|c|c|c|c|c|c}
					\multicolumn{1}{l}{0}&\multicolumn{8}{l}{}\\ \cline{2-4}
					& \multicolumn{1}{l}{\ast} & \multicolumn{2}{l}{\ast}&\multicolumn{5}{l}{}\\ \cline{3-4}
					& \ast & \multicolumn{1}{l}{\ast}\\ \cline{4-6}
					&&& \multicolumn{1}{l}{1} &\multicolumn{2}{l|}{} & \multicolumn{3}{l}{}\\
					\multicolumn{3}{l|}{}&\multicolumn{1}{l}{} & \multicolumn{1}{l}{\ddots}&\multicolumn{1}{l|}{} & \multicolumn{3}{l}{}\\
					\multicolumn{3}{l|}{}&\multicolumn{1}{l}{}&\multicolumn{1}{l}{} & \multicolumn{1}{l|}{1}& &&\\ \cline{4-6}
					\multicolumn{6}{l}{} & \ast & \ast & \\ \cline{6-7}
					\multicolumn{6}{l}{} & \multicolumn{1}{l}{\ast} & \ast & \\ \cline{6-8}
					\multicolumn{8}{l}{} & 0\vphantom{\displaystyle\sum}
				\end{array}
			\right)$\end{small}}
		\]
	\end{enumerate}
\end{rmk}

\begin{rmk}	\label{rm:TrappingState}
	The a priori requirement $\beta_n \neq 0$ for every $n \ge 1$ assures that the transition operator $T_\psi$ is not trivially reducible. If we had $\beta_{n+1} = 0$  for some $n \in \N$ then for an arbitrary state $\psi$ on $M_2$ we would have \mbox{$T_\psi(p_{[0,n]}) = p_{[0,n]}$} and we could restrict our investigation to the restriction of $T_\psi$ to the subalgebras $p_{[0,n]} \B(\H) p_{[0,n]} = \B(p_{[0,n]} \H)$ and $p_{[0,n]}^\bot \B(\H) p_{[0,n]}^\bot = \B(p_{[0,n]}^\bot \H)$.
	For the investigation on the finite-dimensional algebra $\B(p_{[0,n]}\H)$ we refer to \cite{GKL2006}. On the algebra $\B(p_{[0,n]}^\bot \H)$ we again obtain a transition operator of a quantum birth and death process as discussed in this paper for the shifted coefficients $\tilde \alpha_k := \pm \alpha_{n+k+1}$ and $\tilde \beta_k := \pm \beta_{n+k+1}$ ($k\in \N$). We will show later that our prerequisite $\beta_n \neq 0$ assures irreducibility of $T_\psi$ for every faithful state $\psi$ on $M_2$ (see Theorem~\vref{thm:irred_weaklymix}).
\end{rmk}

\begin{expl}[Jaynes-Cummings Interaction]
	\label{expl:micromaser}
	~\\
	In physics the choices $\alpha_n := \cos(g\, \sqrt n)$ and $ \beta_n := -\sin(g \, \sqrt n)$ with some field constant $g \in \R$ are of special interest. For these parameters the transition operator $T_\psi$ describes the interaction of a single-mode electromagnetic field with a two-level atom in the micromaser experiment according to the Jaynes-Cummings interaction (see \cite{JC1963, VBW+2000}). In this setting our a priori assumption $\beta_n \neq 0$ for every $n \in \N$ is known as the absence of the trapping-state condition (\cf~\cite{WVH+1999}). For a diagonal state $\psi$ (temperature state) some ergodic properties of this model, with and without the trapping-state condition, were analyzed by L.\,Bruneau and C.-A.\,Pillet \cite{BP2009}. They characterized the invariant normal states and proved ergodicity, alias thermal relaxation.
\end{expl}

\begin{prop}	\label{prop:KrausDecomp}
	Let $\psi$ be a state on $M_2$ parametrized as in Equation \eqref{eq:density} with \mbox{$0 \le \lambda \le 1$} and $\zeta \in \D$. Then a Kraus decomposition of the transition operator $T_\psi$ is given~by
	\begin{gather*}
		T_\psi(x) = \lambda (t_1^* x t_1 + t_2^* x t_2) + (1-\lambda)(1-\abs{\zeta}^2) (t_3^* x t_3 + t_4^* x t_4)
		\shortintertext{with}
		\begin{aligned}
			t_1 &:= s^* as + i \zeta \sqrt{\tfrac{1-\lambda}{\lambda}} \, s^* b \;,
			&
			t_2 &:= bs - i \zeta \sqrt{\tfrac{1-\lambda}{\lambda}} \, a \;,
			\\
			t_3 &:= s^* b \;, 
			& 
			t_4 &:= a \;.
		\end{aligned}
	\end{gather*}
\end{prop}

\begin{rmk}
	Note that for a faithful state $\psi$ the above Kraus decomposition is minimal if and only if the operators $t_1, \dots, t_4$ are linearly independent. This is exactly the case if $a$ is not a multiple of $s^* a s$, \ie, if there is  no $-1 < q < 1$ with $\alpha_n = q^n$ for all $n \in \N$.
\end{rmk}

\begin{proof}
	For $x \in \B(\H)$ we compute
	\begin{align*}
		T_\psi(x) 
		&= P_\psi \bigl( u^* (x \tensor \one) u \bigr)  
		= P_\psi \begin{pmatrix}
			s^*as \, x \, s^* a s + s^*b \, x \, bs 
			&
			i(s^*as \, x \, s^*b - s^* b \, x \, a)
			\\
			i(a \, x \, bs - bs\, x \, s^*as)
			&
			bs \, x \, s^* b + a \, x \, a
		\end{pmatrix}
		\\
		&= \lambda (s^*as x s^* a s + s^*b x bs ) + (1-\lambda)( bs x s^* b + a x a) + 
		\\
		& \qquad + i \bar\zeta \sqrt{\lambda(1-\lambda)} \,(a x bs - bsx s^*as) + i \zeta \sqrt{\lambda(1-\lambda)}  \, (s^*as x s^*b - s^* b x a)
		\\
		&= \lambda (t_1^* x t_1 + t_2^* x t_2) + (1-\lambda)(1-\abs{\zeta}^2) (t_3^* x t_3 + t_4^* x t_4) \;,
	\end{align*}
	since
	\begin{align*}
		\lambda t_1^* x t_1 
		&= \lambda s^*as \, x \, s^* a s +i \zeta \sqrt{\lambda(1-\lambda)}  \, s^*as \, x \, s^*b - i \bar \zeta \sqrt{\lambda(1-\lambda)} \, bs \, x\, s^* as +
		\\
		&\qquad+ (1-\lambda) \abs{\zeta}^2\, bs  \, x \, s^* b \;,
		\\
		\lambda t_2^* x t_2
		&= \lambda s^* b \, x \, bs - i \zeta \sqrt{\lambda(1-\lambda)} \, s^* b \, x \, a + i \bar\zeta \sqrt{\lambda(1-\lambda)} \, a \, x \, bs + (1-\lambda) \abs{\zeta}^2 axa \;.
		\qedhere
	\end{align*}
\end{proof}

\pagebreak[3]
\begin{prop}	\label{prop:rotation}
	Let $\psi$ be a state on $M_2$ and \mbox{$\theta \in \C$, $\abs \theta = 1$}. Denote by $v_\theta \in M_2$ and $u_\theta \in \B(\H)$ the unitaries
	\begin{align*}
		v_\theta &:= \begin{pmatrix}
			1  \\  & \bar \theta
		\end{pmatrix} \;,
		&
		u_\theta &:= \begin{pmatrix}
			1  \\
			& \theta \\
			&  & \theta^2 \\
			&&& \ddots & 
		\end{pmatrix} 
	\end{align*}
	and set $\psi_\theta := \psi(v_\theta^* \, \cdot \, v_\theta)$. Then for every $x \in \B(\H)$ we have
	\begin{equation*}
		T_{\psi_\theta}(x) 
		= u_\theta\, T_\psi(u_\theta^*\, x\, u_\theta)\, u_\theta^* \;.
	\end{equation*}
	In this sense $T_\psi$ and $T_{\psi_\theta}$ are (inner) conjugate.
\end{prop}

\begin{proof}
	Denote by $u \in \B(\H \tensor \C^2)$ the unitary defined by Equation~\vref{eq:unitary}. Using the commutation relation $s^* u_\theta = \theta u_\theta s^*$ and its adjoint, we find
	\begin{equation*}
		u (u_\theta \tensor v_\theta) = \begin{pmatrix}
			s^* a s u_\theta 	&  i\bar \theta s^* b u_\theta \\
			i bs u_\theta 		& \bar \theta a u_\theta
		\end{pmatrix} = \begin{pmatrix}
			u_\theta s^* a s 		& i u_\theta s^* b \\
			i \bar \theta u_\theta bs	& \bar \theta u_\theta a
		\end{pmatrix} = (u_\theta \tensor v_\theta) u \;.
	\end{equation*}
	Now we compute
	\begin{align*}
		u_\theta^* \, T_{\psi_\theta}(x) \, u_\theta
		&= u_\theta^* \, P_\psi \bigl( (\one \tensor v_\theta^*) \, u^* \, (x \tensor \one) \, u \, (\one \tensor v_\theta) \bigr) \, u_\theta
		\\
		&= P_\psi \bigl( (u_\theta \tensor v_\theta)^* \, u^* \, (x \tensor \one) \, u \, (u_\theta \tensor v_\theta)  \bigr)
		\\
		&= P_\psi \bigl( u^* (u_\theta^* x u_\theta \tensor \one) u \bigr)
		= T_\psi(u_\theta^* x u_\theta) \;.
		\qedhere
	\end{align*}
\end{proof}

\begin{rmk}\label{rmk:rotation}
	\begin{enumerate}
	\item 
		If the state $\psi$ is parametrized as in Equation~\ref{eq:density} with $0\le \lambda \le 1$ and \mbox{$\zeta \in \D$} then the state $\psi_\theta$ is given by the parameters $\lambda$ and $\zeta \cdot \theta$. Thus, by Proposition~\ref{prop:rotation} we may restrict our investigations to states $\psi$ with $\zeta \ge 0$ up to a suitable unitary conjugation of the transition operator.
	\item
		Observe that replacing  $\zeta$ by $-\zeta$ in the parametrization of the state $\psi$ is equivalent to replacing $u$ by $u^*$ (or $b$ by $-b$) in the construction of the transition operator. It hence follows from Proposition~\ref{prop:rotation} that the time reversed process with transition operator
		\begin{equation*}
			T_\psi^+(x) := P_\psi \bigl( u(x\tensor \one)u^* \bigr)
		\end{equation*}
		also belongs to our class of quantum birth and death chains. More precisely, choosing $\theta = (-1)$ in Proposition~\ref{prop:rotation} we obtain $T_\psi^+(x) = u_{-1} \, T_\psi(u_{-1}^* \, x \, u_{-1}) \, u_{-1}^* = T_{\psi_{-1}}(x)$. 
	\end{enumerate}
\end{rmk}

\begin{thm}	\label{thm:extremal}
	\begin{enumerate}
		\item 
			Let $\psi \neq \psi_+$ be a state on $M_2$. Then $T_\psi$ is extremal in $\UCP(\H)$ if and only if $\psi$ is pure.
		\item If $\psi=\psi_+$ then the following statements are equivalent:
		\begin{enumerate}
			\item 
				The transition operator $T_{\psi_+}$ is \emph{not} extremal in $\UCP(\H)$.
			\item	\label{en:extremal_2b}
				$s^* a^2 s$ is a non-zero multiple of $\one$ or a non-zero multiple of $p_0$.
		\end{enumerate}
	\end{enumerate}
\end{thm}

\begin{proof}
	Let $\psi$ be parametrized by $0 \le \lambda \le 1$ and $\zeta \in \D$ as in Equation~\ref{eq:density}. We denote by $t_1, t_2, t_3, t_4$ the Kraus operators given in Proposition~\vref{prop:KrausDecomp}. 
	\begin{enumerate}
	\item 
		If $\psi$ is not pure, \ie $0 < \lambda < 1$ and $\abs\zeta < 1$, then 
		\begin{align*}
			T_1(x) &:= t_3^* x t_3 + t_4^* x t_4 \;,
			&
			T_2(x) &:= \lambda (\lambda + \abs{\zeta}^2 - \lambda \abs{\zeta}^2)^{-1} (t_1^* x t_1 + t_2^* x t_2) 
		\end{align*}
		are ucp\nobreakdash-maps. Since $\scal{T_1(p_0) \, e_1, e_0} = 0 \neq \scal{T_2(p_0) \, e_1, e_0}$ these maps do not coincide (\cf Appendix~\ref{appendix}). Being a non-trivial convex combination of $T_1$ and $T_2$, the transition operator $T_\psi$ is not extremal in $\UCP(\H)$.

		For the converse, recall that in general a ucp\nobreakdash-map $T(x) = \sum_{i=1}^N a_i^* x a_i$ on $\BH$ is extremal in $\UCP(\H)$ if the set $\{a_{i}^* a_{j} \;|\; 1 \le i,j \le N\}$ is linearly independent. We first consider $\psi = \psi_-$. Then the transition operator $T_\psi$ is given by
		\begin{equation*}
			T_{\psi_-}(x) = t_3^* x t_3 + t_4^* x t_4 = bs \, x \, s^* b + a \, x \, a  
			\;.
		\end{equation*}
		Hence we have to consider the products
		\begin{align*}
			t_3^* t_3 &= b^2 \;, 
			& 
			t_4^* t_4 &= a^2 \;,
			&
			t_3^* t_4 &= s(a s^*bs) \;,
			&
			t_4^* t_3 &= (a s^* bs) s^* \;.
		\end{align*}
		Observe that $a^2$ and $b^2$ are linearly independent diagonal operators ($\alpha_0 = 1$, \mbox{$\beta_0 = 0$}). Moreover, since $as^*bs$ is a non-zero diagonal operator ($\alpha_0=1$, $\beta_1\neq 0$), the matrix of coefficients of $t_3^* t_4$ and $t_4^* t_3$ is a non-vanishing strictly lower and strictly upper triangular matrix, respectively. Consequently, the above products are indeed linearly independent.

		Now let $\psi \notin \{\psi_+,\psi_-\}$ be pure, \ie $0 < \lambda < 1$ and $\abs \zeta = 1$. By Proposition~\ref{prop:rotation} we may assume $\zeta =1$. Then the transition operator is given by
		\begin{equation*}
			T_\psi(x) = \lambda (t_1^* x t_1 + t_2^* x t_2) \;.
		\end{equation*}
		Hence we have to show that the products
		\begin{small}
		\begin{align*}
			\lambda t_1^* t_1 &= -i \sqrt{\lambda(1-\lambda)} abs + \lambda s^* a^2 s + (1-\lambda) b^2 + i \sqrt{\lambda(1-\lambda)} s^* ab \;,
			\\
			\lambda t_2^* t_2 &= i \sqrt{\lambda(1-\lambda)} abs + \lambda s^* b^2 s + (1-\lambda) a^2 - i \sqrt{\lambda(1-\lambda)} s^* ab \;,
			\\
			\lambda t_1^* t_2 &= -i\sqrt{\lambda(1-\lambda)} (bsbs^*)s^2 + \lambda (s^* as b)s - (1-\lambda)(s as^* b)s - i \sqrt{\lambda(1-\lambda)} s^* as a \;,
			\\
			\lambda t_2^* t_1 &= i \sqrt{\lambda(1-\lambda)} as^* as + \lambda s^*(s^* as b) - (1-\lambda) s^* (sas^* b) + i\sqrt{\lambda(1-\lambda)} (s^*)^2 (bsbs^*) \;.
		\end{align*}
		\end{small}
		\!\!%
		are linearly independent. Since $\beta_n \neq 0$ for $n \ge 1$, the functional $\varphi_{0,2}(x) := \scal{x \, e_0, e_2}$ on $\BH$ vanishes on all these products except $\lambda t_1^* t_2$ and, analogously, $\varphi_{2,0}(x) := \scal{x \, e_2, e_0}$ vanishes on all products except $\lambda t_2^*t_1$. It therefore suffices to prove that $\lambda t_1^* t_1$, $\lambda t_2^* t_2$ are linearly independent, \ie, they generate a two-dimensional subspace. Obviously, the linear span of $\lambda t_1^* t_1$ and $\lambda t_2^* t_2$ contains the identity operator $\one$. Assume $\lambda t_1^* t_1 \in \C \one$. Then $ab$ must vanish and, since $\beta_n \neq 0$ for all $n \ge 1$, this implies $a = p_0$. But then $\lambda t_1^* t_1 = (1-\lambda) b^2\notin\C\one$. Hence the linear span of $\lambda t_1^* t_1$ and $\lambda t_2^* t_2$ is indeed two-dimensional.
	\item
		Let $\psi = \psi_+$, \ie $\lambda=1$ and $\zeta = 0$.
		First suppose $s^* a^2 s = \mu \one$ for some $\mu \neq 0$. Then also $\mu \neq 1$, since $0 \neq s^* b^2 s = (1-\mu) \one$. Hence the transition operator
		\begin{equation*}
			T_{\psi_+}(x) = \mu( \mu^{-1} s^* as \, x \, s^* as) + (1-\mu) \bigl( (1-\mu)^{-1} s^* b \, x \, bs \bigr)
		\end{equation*}
		is a convex combination of two distinct ucp\nobreakdash-maps.
		Second, if $s^* a^2 s = \mu p_0$ for some $\mu \neq 0$ then
		\begin{equation*}
			T_{\psi_+}(x) = \tfrac12 (\sqrt\mu\, p_0 - s^* b) \, x\, (\sqrt\mu\, p_0 - bs) + \tfrac12 (\sqrt\mu\, p_0 + s^* b) \, x \, (\sqrt\mu\, p_0 + bs)
		\end{equation*}
		is a convex decomposition into distinct ucp\nobreakdash-maps.

		Conversely, if $s^* as = 0$ then $T_{\psi_+}(x) = s^*b \, x \, bs$ is obviously extremal. Otherwise, the transition operator is given by
		\begin{equation*}
			T_{\psi_+}(x) = t_1^* x t_1 + t_2^* x t_2 = s^* as \, x \, s^* as + s^* b\, x \, bs \;.
		\end{equation*}
		We show that the products
		\begin{align*}
			t_1^* t_1 &= s^* a^2 s \;, 
			& 
			t_2^* t_2 &= s^* b^2 s = \one - s^* a^2 s \;,
			\\
			t_1^* t_2 &= (s^* as b) s \;,
			&
			t_2^* t_1 &= s^* (s^* as b) 
		\end{align*}
		are linearly independent. If $s^* a^2 s$ (and hence $s^* as$) is not a multiple of $p_0$ then the diagonal operator $s^* as b$ does not vanish. It follows that for $t_1^* t_2$ and $t_2^* t_1$ the matrix of coefficients is a non-zero strictly lower and strictly upper triangular matrix, respectively. Moreover, if $s^* a^2 s$ is not a multiple of $\one$ then $t_1^* t_1$ and $t_2^* t_2$ are linearly independent diagonal operators. Consequently, if Condition~\ref{en:extremal_2b} is not satisfied then $T_{\psi_+}$ is extremal in $\UCP(\H)$.\qedhere
	\end{enumerate}
\end{proof}

\begin{rmk}
	The above proof relies on the assumption that $\beta_n\neq0$ for all $n\in\N$. Dropping this assumption we find the following counterexample: Let \mbox{$a:=\diag(1,0,1,0,1,\ldots)$}, \mbox{$b:=\diag(0,1,0,1,0,\ldots)$} and let $\psi$ be the state on $M_2$ with parameters $\lambda:=\nicefrac12$ and \mbox{$\zeta:=1$}. Then $\psi$ is pure and a convex decomposition of the transition operator $T_\psi$ into two distinct ucp\nobreakdash-maps is given by ($x\in\BH$)
	\begin{equation*}
		T_\psi(x)=\tfrac12(t_1^*xt_1+t_2^*xt_2)=\tfrac12\bigl(\tfrac12(t_1+t_2)^*x(t_1+t_2)\bigr)+\tfrac12\bigl(\tfrac12(t_1-t_2)^*x(t_1-t_2)\bigr) \;.
	\end{equation*}
\end{rmk}

\section{General Ergodic Properties}	\label{sec:ergProp}
Now we shift our attention to ergodic properties of the transition operator $T_\psi$. More precisely, we consider the following notions of ergodic theory: 
Let $T:\BH\to\BH$ be a ucp\nobreakdash-map. A projection $p\in\BH$ satisfying $T(p) \ge p$ is called \emph{subharmonic}. The map $T$ is called \emph{irreducible} if there are no subharmonic projections except $0$ and $\one$. If the fixed space of $T$, which we denote by $\fix T$, consists only of multiples of $\one$ then $T$ is called \emph{ergodic}. If the multiples of $\one$ are the only eigenvectors of $T$ corresponding to any eigenvalue $\mu\in\C$ with $\abs{\mu}=1$ then $T$ is called \emph{weakly mixing}.

\pagebreak[3]
\begin{prop}	\label{prop:subh}
	Let $T:\BH\to\BH$ be a ucp\nobreakdash-map and fix a Kraus decomposition $T(x) = \sum_{i \in I} a_i^*xa_i$ with $a_i \in \BH$, $i \in I$.
	\begin{enumerate}
		\item 
			For a projection $p\in\BH$ the following conditions are equivalent:
			\begin{enumerate}
				\item $p$ is subharmonic.
				\item\label{item:a_ip=pa_ip} $a_ip = pa_ip$ for every $i \in I$.
			\end{enumerate}
		\item 	\label{en:subh_eigval}
			Suppose that $T$ admits a faithful invariant state and let $\mu\in\C$, $\abs{\mu}=1$. Then for $x\in\BH$ the following conditions are equivalent:
			\begin{enumerate}
				\item $T(x)=\mu\,x$.
				\item $xa_i=\mu\,a_ix$ for every $i\in I$.
			\end{enumerate}
	\end{enumerate}
\end{prop}

\pagebreak[3]
\begin{rmk}	\label{rmk:subh}
	\begin{enumerate}
		\item If $\mbox{span}\{a_i:i \in I\}\subseteq\BH$ is a self-adjoint set then Condition \ref{item:a_ip=pa_ip} above is equivalent to $a_ip = pa_i$ for every $i\in I$. Hence, if this condition is fulfilled then $p$ is in fact a fixed point, because $T(p)=\sum_{i \in I} a_i^*pa_i=\sum_{i \in I} a_i^*a_i\,p=p$.
		\item 
			In general the fixed space $\fix T$ is not closed under multiplication and therefore not an algebra. However, it obviously contains the commutant $\{a_i, a_i^* : i \in I\}'$. In fact, this is the largest von Neumann subalgebra contained in $\fix T$ (see \cite{AGG2002}). If $T$ admits a faithful invariant state, it follows from Statement~\ref{en:subh_eigval} that the fixed space agrees with $\{a_i, a_i^*: i \in I\}'$; in particular, $\fix T$ is a von Neumann subalgebra (\cf~\cite{Frigerio1978}, \cite{KN1979}).
	\end{enumerate}
\end{rmk}

The following proof is adapted from \cite[Thm.~III.1]{FR2002} (\cf also \cite[Thm.~2.7]{Mohari2005}). We include it for convenience of the reader.

\begin{proof}
	\begin{enumerate}
	\item 
		If $p$ is subharmonic then $p^\bot \geq T(p^\bot) = \sum_{i \in I} a_i^*p^\bot a_i$. Hence we have
		\begin{equation*}
			0 = p\,T(p^\bot)\,p = \sum_{i \in I} (p^\bot a_i p)^*(p^\bot a_i p).
		\end{equation*}
		It follows that $p^\bot a_i p$ must vanish for each $i \in I$ and, therefore, $a_ip = p^\bot a_ip + pa_ip = pa_ip$.

		Conversely, if $a_ip = pa_ip$ then $T(p^\bot)\,p = \sum_{i \in I} a_i^*p^\bot a_i p = \sum_{i \in I} a_i^*p^\bot p a_i p = 0$ and $pT(p^\bot)=(T(p^\bot)p)^*=0$. Hence we obtain
		\begin{equation*}
			T(p^\bot)= p^\bot T(p^\bot) p^\bot + p^\bot T(p^\bot) p + p T(p^\bot) = p^\bot T(p^\bot) p^\bot \leq p^\bot,
		\end{equation*}
		\ie, $p$ is subharmonic.
	\item 
		Let $\varphi$ be a faithful invariant state for $T$ and $T(x)=\mu x$. Then by the Kadison-Schwarz inequality $0\leq\varphi \bigl( T(x^*x)-T(x)^*T(x) \bigr) = \varphi(x^*x - \bar\mu\mu\,x^* x)=0$ and hence $T(x^*x)=T(x)^*T(x)$. It follows
		\begin{align*}
			0 &= (T(x) - \mu x)^*(T(x) - \mu x) = T(x^*x) - \mu T(x)^*x - \bar\mu x^*T(x) + x^*x \\
			&= \sum\nolimits_{i\in I} a_i^*x^*xa_i-\mu a_i^*x^*a_ix-\bar\mu x^*a_i^*xa_i+x^*a_i^*a_ix\\
			&= \sum\nolimits_{i\in I} (xa_i - \mu a_ix)^*(xa_i - \mu a_ix).
		\end{align*}
		Therefore, $(xa_i - \mu a_ix)$ vanishes for each $i \in I$, \ie $xa_i = \mu a_i x$. The converse implication is trivial. \qedhere
	\end{enumerate}
\end{proof}

\begin{thm}\label{thm:irred_weaklymix}
	Let $\psi$ be a faithful state on $M_2$. Then we have:
	\begin{enumerate}
		\item 	\label{en:irred_weaklymix_irred}
			The transition operator $T_\psi$ is irreducible.
		\item If $T_\psi$ admits an invariant normal state then $T_\psi$ is weakly mixing.
	\end{enumerate}
\end{thm}

We will see in the next section that $T_\psi$ is not always weakly mixing if $\psi$ is faithful. On the other hand, the transition operator $T_\psi$ can be weakly mixing even though it does not admit an invariant normal state (\cf Proposition \vref{prop:infant_mixing} and Theorem \vref{thm:babymaser_mixing}).

\begin{proof}
	\begin{enumerate}
	\item 
		Let $p \in \BH$ be a subharmonic projection. It is easily checked that the Kraus operators $t_1, \dots, t_4$ given in Proposition~\vref{prop:KrausDecomp} span a self-adjoint linear subspace. Hence by Proposition~\ref{prop:subh} and Remark~\ref{rmk:subh} we have $T_\psi(p) = p$ and $p$ commutes with $t_3 = s^* b$ and $t_3^* = bs$. It follows that $p$ also commutes with the operator
		\begin{equation*}
			(t_3^*)^m t_3^m = (bs)^m (s^* b)^m = \begin{pmatrix}
				0 \\
				& \ddots \\
				&& 0 \\
				&&& \prod_{k=1}^m \beta_k^2 \\
				&&&& \prod_{k=2}^{m+1} \beta_k^{2}\\
				&&&&& \ddots
			\end{pmatrix}
		\end{equation*}
		for each $m\ge 1$ and therefore with its kernel projection $p_{[0,m-1]}$. We conclude that \mbox{$p \, p_m = p_m \, p$} for every $m \in \N$, \ie, $p$ is a diagonal operator.

		Let $P_{\ell^\infty(\N)}(x) := \sum_{n=0}^\infty p_n x p_n$, $x \in \B(\H)$, denote the conditional expectation onto the commutative subalgebra $\ell^\infty(\N) \subseteq \B(\H)$ of all diagonal operators. Due to $T_\psi(p) = p \in \ell^\infty(\N)$ the projection $p$ is a fixed point of the classical transition operator $P_{\ell^\infty(\N)} \circ T_\psi|_{\ell^\infty(\N)}$ (\cf Remark~\ref{rm:whyBnD} and Figure~\vref{fig:classical_bnd_chain}).
		Since $\beta_n \neq 0$ for every $n \ge 1$ and $\psi$ is faithful ($0 < \lambda < 1$), this classical transition operator is  irreducible and its only fixed projections are $0$ and $\one$ (\cf, \eg, \cite[Sec.5.3]{Durrett1996}).
		Consequently, either $p = 0$ or $p =\one$.
	\item
		Let $T_\psi(u) = \mu u$ for some $u \in \BH$, $\norm u = 1$,  and some $\mu \in \C$, $\abs \mu = 1$. We first show that the existence of an invariant normal state $\varphi$ implies that $u$ is unitary. Since the support projection of any invariant normal state is subharmonic, $\varphi$ is faithful due to \ref{en:irred_weaklymix_irred}. By Remark \ref{rmk:subh} the fixed space $\mathcal F(T_\psi)$ is a von Neumann subalgebra. In particular, it is generated by its projections, which, again by \ref{en:irred_weaklymix_irred}, implies $\mathcal F(T_\psi) = \C \one$. By virtue of the Kadison-Schwarz inequality the element $T_\psi(u^* u) - T_\psi(u)^* T_\psi(u) = T_\psi(u^* u) - u^* u$ is positive. Moreover, it satisfies
		\begin{equation*}
			0 \le \varphi \bigl( T_\psi(u^* u) - u^* u \bigr) = (\varphi \circ T_\psi)(u^* u) - \varphi(u^* u) = 0 \;.
		\end{equation*}
		Since $\varphi$ is faithful, we conclude that $T(u^* u) = u^* u$, \ie, $u^* u$ is an element of the fixed point space~$\mathcal F(T_\psi) = \C \one$. An analogous argument shows $u u^* \in \C \one$, \ie, $u$~is unitary.

		To conclude that $u \in \C \one$ we now show that $u$ is in fact a fixed point, \ie $\mu = 1$. By Proposition~\ref{prop:subh}.\ref{en:subh_eigval} we have $ut_i = \mu t_i u$ for all $1 \le i \le 4$. It is easily seen that \mbox{$t_3^* \in \lin\{t_1, \dots, t_4\}$}. Hence also $ut_3^* = \mu t_3^* u$ and 
		\begin{equation*}
			u b^2 = u t_3^* t_3 = \mu^2 t_3^* t_3 u = \mu^2 b^2 u \;.
		\end{equation*}
		This implies that $u$ commutes with $p_0$, the kernel projection of $b^2$. It follows that
		\begin{equation*}
			p_0 = u p_0 u^* = u t_4 p_0 u^* = \mu t_4 p_0 = \mu p_0  \;;
		\end{equation*}
		consequently, $\mu = 1$.\qedhere
	\end{enumerate}
\end{proof}

Now, let $\psi$ be a diagonal state on $M_2$. Then the transition operator
simplifies to 
\begin{equation*}
	T_\psi(x)=\lambda \bigl( s^*as\,x\,s^*as+s^*b\,x\,bs \bigr) + (1-\lambda)\bigl(bs\,x\,s^*b+a\,x\,a\bigr), 
	\quad\  x \in \BH\;.
\end{equation*}
Recall that in this case the algebra $\ell^\infty(\N)$ of all diagonal operators is invariant for $T_\psi$ and for each diagonal normal state $\varphi$ on $\BH$ the state $\varphi \circ T_\psi$ is diagonal, too (\cf~Remark~\vref{rm:whyBnD}). In this case we may extend some results from the classical subchain on $\ell^\infty(\N)$ without much effort.

\begin{prop}	\label{prop:invariantState_zeta=0}
	Let $\psi$ be a diagonal state on $M_2$ parametrized as in Equation \eqref{eq:density} with \mbox{$0 \le \lambda \le 1$} and $\zeta=0$. 
	Then $T_\psi$ admits an invariant normal state $\varphi$ on $\BH$ if and only if $\lambda < \nicefrac12$. In this case $\varphi$ is given by $\varphi=\Tr(\rho\;\cdot\,)$ with
	\begin{equation*}
		\rho=\frac{1-2\lambda}{1-\lambda}\begin{pmatrix}1\\ &\tfrac{\lambda}{1-\lambda}\\ &&\left(\tfrac{\lambda}{1-\lambda}\right)^2\\ &&& \ddots \end{pmatrix}.
	\end{equation*}
\end{prop}

\begin{proof}
	The proof essentially relies on the same computations as for the classical birth and death chain.
	Using $a^2+b^2=\one$ and $s^*b\rho bs=(s^*b^2s)(s^*\rho s)=(s^*b^2s)\tfrac{\lambda}{1-\lambda}\rho$ we obtain
	\begin{align*}
		\lambda& s^*as\rho s^*as +\lambda bs\rho s^*b+(1-\lambda)s^*b \rho bs+(1-\lambda)a \rho a\\
		&= \lambda s^*a^2s\rho+\lambda b^2\tfrac{1-\lambda}{\lambda}\rho +(1-\lambda)s^*b^2s \tfrac{\lambda}{1-\lambda}\rho +(1-\lambda)a^2 \rho\\
		&=\lambda s^*(a^2+b^2)s\rho+(1-\lambda)(a^2+b^2)\rho=\rho.
	\end{align*}
	This shows that $\varphi=\Tr(\rho\;\cdot\,)$ is indeed invariant for $\lambda<\nicefrac12$.

	Conversely, let $\lambda\geq\nicefrac12$ and assume that $\varphi = \Tr(\rho \, \cdot\,)$ is an invariant normal state. Then the coefficients $\rho_{n,n}$, $n\in\N$, satisfy (\cf Figure \vref{fig:action})
	\begin{align*}
		\rho_{0,0}&=(\lambda\alpha_1^2+(1-\lambda))\rho_{0,0}+(1-\lambda)\beta_1^2\rho_{1,1},\\
		\rho_{n,n}&=(\lambda\alpha_{n+1}^2+(1-\lambda)\alpha_{n}^2)\rho_{n,n}+(1-\lambda)\beta_{n+1}^2\rho_{n+1,n+1}+\lambda\beta_{n-1}^2\rho_{n-1,n-1}.
	\end{align*}
	For $\lambda=1$ this implies $\rho_{n,n}=0$ for every $n\in\N$, since $\alpha_n^2 < 1$ for every $n\geq1$. Using $\alpha_n^2+\beta_n^2=1$, for $\nicefrac12\leq\lambda<1$ a simple induction shows that 
	\begin{equation*}
		\rho_{n,n}=\left(\tfrac{\lambda}{1-\lambda}\right)^n\rho_{0,0}
	\end{equation*}
	for every $n\in\N$. But by normality we then obtain the contradiction $1=\varphi(\one)=\sum_{n=0}^\infty\rho_{n,n} \in \{0, \infty\}$. Hence there is no invariant normal state.
\end{proof}

\begin{expl}[Ergodic Properties of $T_{\psi_+}$ and $T_{\psi_-}$]
	\label{expl:psi+-}
	~\\	
	As special cases we examine the transition operators corresponding to the states $\psi_+$ and~$\psi_-$, that is, in terms of the parametrization of Equation~\vref{eq:density} we consider $\lambda=1$ and $\lambda = 0$, respectively.

	Since $T_{\psi_+}(p_0^\perp) = \one - \alpha_1^2\, p_0 \ge p_0^\perp$ and $T_{\psi_-}(p_0) = \beta_1^2\, p_1 + p_0 \ge p_0$, both transition operators admit a non-trivial subharmonic projection and are therefore \emph{not irreducible}.

	We have shown in Proposition~\ref{prop:invariantState_zeta=0} that $T_{\psi_+}$ does not admit an \emph{invariant normal state}. For the transition operator $T_{\psi_-}$ the vector state $\varphi = \Tr(p_0 \, \cdot\,) = \scal{\, \cdot \, e_0, e_0}$ is invariant.

	In general, $T_{\psi_+}$ is \emph{not ergodic} (and hence \emph{not weakly mixing}). For an example with an infinite dimensional fixed space we refer to Proposition~\ref{prop:infant_mixing} or Theorem~\ref{thm:babymaser_mixing} in the next section.
	However, for certain choices of model parameters $\alpha_n$ and $\beta_n$, $n \ge 1$, the transition operator $T_{\psi_+}$ is weakly mixing. For instance, let $\beta_n := (\tfrac12)^n$ and $\alpha_n := \sqrt{1 - \beta_n^2}\,$ ($ n \ge 1$). Then for $x  \in \B(\H)$ such that $T_{\psi_+}(x) = \mu x$ with $\abs \mu = 1$ the coefficients of $x$ satisfy the equations
	\begin{align*}
		x_{m+1,n+1} &= 2^{m+n+2} \; \Bigl( \mu - \sqrt{ \bigl( 1 - (\tfrac12)^{2m+2} \bigr) \bigl( 1 - (\tfrac12)^{2n+2} \bigr)}\,  \Bigr) \; x_{m,n}
	\end{align*}
	for all $m,n \in \N$. We focus on the factor in this recursion: Using the fact that the arithmetic mean is always greater than the geometric mean, we obtain for $n := m+k$ with $k \ge 1$:
	\begin{equation*}
		\abs[\Big]{2^{2m+k+2} \Bigl( \mu - \sqrt{ \bigl( 1- (\tfrac12)^{2m+2} \bigr) \bigl( 1 - (\tfrac12)^{2m+2k+2} \bigr)}\, \Bigr)}
		\ge 2^{k-1} + 2^{-k-1} \ge \tfrac54.
	\end{equation*}
	Since the coefficients of $x \in \B(\H)$ are bounded, we must have $x_{m,n} = 0$ for all $m < n$ ($k \ge 1$) and, analogously, $x_{m,n} = 0$ for all $m > n$. For $m=n$ the recursion factor is given by
	\begin{equation*}
		2^{2m+2} \Bigl( \mu - \bigl( 1-(\tfrac12)^{2m+2} \bigr) \Bigr)
		= 2^{2m+2} (\mu-1) + 1\;.
	\end{equation*}
	This shows that, if $\mu \neq 1$, we have $x_{m,m} = 0$ for all $m \in \N$, \ie $x = 0$. If $\mu = 1$ then $x = x_{0,0} \cdot \one$ follows.
	
	For the state $\psi_-$ the transition operator $T_{\psi_-}$ is \emph{weakly mixing} (and hence \emph{ergodic}) for any choice of model parameters $\alpha_n$ and $\beta_n$. For if $x\in\B(\H)$ and $T_{\psi_-}(x) = \mu x$ with $\abs \mu =1$, the coefficients of $x$ satisfy the equations ($m,n \in \N$)
	\begin{gather*}
		(\mu - \alpha_m \alpha_n) x_{m,n} = \begin{cases}
			\beta_m \, \beta_n \, x_{m-1, n-1}\,, &\text{if $m,n \ge 1$,}
			\\
			0\,, &\text{otherwise.}
		\end{cases}
	\end{gather*}
	Unless $m=n=0$ the factors $(\mu - \alpha_m \alpha_n)$ and $\beta_m \beta_n$ for $m,n \ge 1$ do not vanish. Hence $x_{m,n} = 0$ for all $m \neq n$. If $\mu \neq 1$, this also follows for all $m=n$, \ie $x = 0$. If $\mu=1$ then $x_{m,m} = \frac{1-\alpha_m^2}{\beta_m^2} \, x_{m-1, m-1} = x_{m-1, m-1}$ for all $m \ge 1$, \ie $x \in \C \one$.
\end{expl}

\section{Toy Examples}	\label{sec:toy}

In this section we consider the special model parameters $\alpha_n := \alpha$ and $\beta_n := \beta$, $n \ge 1$, for some fixed numbers $-1 \le \alpha, \beta \le 1$ with $\beta \neq 0$ and $\alpha^2 +\beta^2 = 1$. This choice corresponds to homogeneous birth and death rates in the sense that for every state $\psi$ on $M_2$ the transition probability
\begin{equation*}
	\scal{T_\psi(p_m) e_n, \, e_n}
\end{equation*}
depends only on the difference $m-n$, unless $m=n=0$. This means that the classical Markov chain obtained by restricting $P_{\ell^\infty(\N)} \circ T_\psi$ to the subalgebra of diagonal operators is indeed a homogeneous birth and death chain (\cf Remark~\ref{rm:whyBnD} and Figure~\vref{fig:classical_bnd_chain}). 

Limiting the choice of model parameters to $\alpha_n = \alpha$ and $\beta_n = \beta$, $n \ge 1$,  allows us to determine precisely for which states $\psi$ the transition operator $T_\psi$ admits a pure invariant normal state. Notice that by Theorem~\vref{thm:irred_weaklymix} this can only happen if $\psi$ is pure, too. Moreover, for diagonal states we obtain a complete characterization of ergodicity and weak mixing of~$T_\psi$.

Later in this section we will restrict even more to $\alpha_n := 0$ and $\beta_n:=1$, $n \ge 1$. Because of its simplicity we refer to this choice as the \emph{baby maser}. In this case the classical birth and death chain obtained by restricting to the algebra of diagonal operators is homogeneous without (non-trivial) self-loops.
For this relatively simple model we obtain a complete characterization of invariant normal states, ergodicity, and weak mixing of the transition operator $T_\psi$ for \emph{every} state $\psi$ on $M_2$.

\begin{prop}	\label{prop:infant_pureInvariantState}
	Let $\alpha_n=\alpha$ and $\beta_n=\beta\neq 0$ for all $n\geq1$ and let $\psi$ be a pure state on $M_2$ parametrized as in Equation \eqref{eq:density} with 
	$0 < \lambda < 1$ and $\zeta \in \C$, $\abs \zeta = 1$.
	\\
	Then the transition operator $T_\psi$ admits a pure invariant normal state 
	if and only if $\lambda < \tfrac12(1-\alpha)$.
	In this case there is only one pure invariant normal state $\varphi$ on $\BH$ and it is given by
	 \begin{equation*}
		 \varphi(x)=\scal{x\xi,\xi}
		 \quad
		 \text{with}
		 \quad 
		 \xi= \sqrt{1-\abs{q}^2} \cdot (q^n)_{n\in\N}
		 \quad
		 \text{and}
		 \quad 
		 q=-i\zeta\tfrac{\beta}{1-\alpha}\sqrt{\tfrac{\lambda}{1-\lambda}}.
	 \end{equation*}
\end{prop}

\begin{rmk}
	For $\lambda \in \{0,1\}$, \ie for the states $\psi_-$ and $\psi_+$, we already discussed invariant normal states for general model parameters in Proposition~\vref{prop:invariantState_zeta=0} (see also Example~\ref{expl:psi+-}). Note that the conclusion of Proposition~\ref{prop:infant_pureInvariantState} extends to these two cases. 
\end{rmk}

\begin{proof}
	Let $0<\lambda<1$ and $\zeta\in\C$, $\abs\zeta=1$, and observe that in this case only the two operators
	\begin{align*}
		t_1 &= \alpha \one + i \zeta \sqrt{\tfrac{1-\lambda}{\lambda}} \, \beta s^*
		&
		&\text{and}
		&
		t_2 &= \beta s + i \zeta \sqrt{\tfrac{1-\lambda}{\lambda}} \, a
	\end{align*}
	contribute to the Kraus decomposition of $T_\psi$ in Proposition~\vref{prop:KrausDecomp}. 

	First let $\xi\in\ell^2(\N)$ be a unit vector such that $\varphi:=\scal{\,\cdot\;\xi,\xi}$ is invariant for $T_\psi$. Then $\xi$ is an eigenvector of $t_1$ and $t_2$. In particular, $\xi$ is an eigenvector of $s^*$. Therefore, we have $\xi=\sqrt{\smash[b]{1-\abs{q}^2}} \cdot (q^n)_{n\in\N}$ for some $q\in\C$, $\abs{q} < 1$. Let $\mu$ be the corresponding eigenvalue of~$t_2$. Then $\mu\scal{\xi,e_0}=\scal{t_2\xi,e_0}$ yields $\mu=i\bar\zeta\sqrt{\tfrac{1-\lambda}{\lambda}}$. Moreover, from
	\begin{align}
		\mu q^n\scal{\xi,e_0}=\mu\scal{\xi,e_n}=\scal{t_2\xi,e_n}=\left(\beta q^{n-1}+i\bar\zeta\sqrt{\tfrac{1-\lambda}{\lambda}}\,\alpha q^n\right)\scal{\xi,e_0}\label{eq:prop:infant_pureInvariantState}
	\end{align}
	for $n\geq 1$, it follows that $q=-i\zeta\tfrac{\beta}{1-\alpha}\sqrt{\tfrac{\lambda}{1-\lambda}}$.
	Since $\alpha^2+\beta^2=1$ and $\abs{q}^2<1$, a~straightforward computation shows that $\lambda<\tfrac{1}{2}(1-\alpha)$.

	Conversely, if $\lambda<\tfrac{1}{2}(1-\alpha)$ then, obviously, the unit vector $\xi$ given in the Proposition is an eigenvector of $s^*$ and hence of $t_1$. Furthermore, for this choice of $q$ Equation \eqref{eq:prop:infant_pureInvariantState} is valid for every $n \ge 1$. This shows that $\xi$ is an eigenvector of $t_2$, too. Consequently, the pure normal state $\varphi=\scal{\,\cdot\;\xi,\xi}$ is invariant.
\end{proof}

Now we turn to mixing properties and ergodicity of the transition operator.
For a faithful diagonal state $\psi$ on $M_2$ with $0<\lambda < \nicefrac12$ and $\zeta=0$ we have already shown in Proposition~\vref{prop:invariantState_zeta=0} that there exists an invariant normal state for $T_\psi$. By Theorem~\vref{thm:irred_weaklymix} it follows that $T_\psi$ is weakly mixing. For our toy example we are able to show that weak mixing also holds for $\lambda = \nicefrac12$ and that for $\lambda > \nicefrac12$ the transition operator is not even ergodic.

\begin{prop}	\label{prop:infant_mixing}
	Let $\alpha_n=\alpha$ and $\beta_n=\beta\neq 0$ for all $n\geq1$ and let $\psi$ be a diagonal state on $M_2$ parametrized as in Equation \eqref{eq:density} with $0 \le \lambda \le 1$ and $\zeta=0$.
	\\	
	Then the transition operator $T_\psi$ is weakly mixing (hence ergodic) if and only if $\lambda\leq\nicefrac12$. For $\lambda>\nicefrac12$ the fixed space of $T_\psi$ is infinite dimensional.
\end{prop}

\begin{proof}
	For $0<\lambda < \nicefrac12$ the assertion follows from Theorem~\ref{thm:irred_weaklymix} and Proposition~\ref{prop:invariantState_zeta=0}. The case $\lambda=0$ (\ie $\psi=\psi_-$) was treated in Example~\vref{expl:psi+-}. Here we only deal with $\lambda = \nicefrac12$ in the first part and with $\lambda > \nicefrac12$ in the second part of the proof.
	\begin{enumerate}
	\item 
		Let $\lambda = \nicefrac12$ and let $x \in \B(\H)$ with $T_\psi(x) = \mu x$ for some $\mu \in \C$, $\abs \mu = 1$. Then the coefficients of $x$ satisfy the following equations (\cf Figure \vref{fig:action}):
		\begin{align}
			\mu\, x_{0,0} &= \tfrac12(\alpha^2 +1) x_{0,0} + \tfrac12 \beta^2 x_{1,1}	
			\label{eq:prop:infant_mixing:1}
			\\
			\mu\, x_{0,k} &= \tfrac12(\alpha^2 +\alpha) x_{0,k} + \tfrac12 \beta^2 x_{1,k+1}, \quad k \ge 1,
			\label{eq:prop:infant_mixing:2}
			\\
			\mu\, x_{m,n}&= \alpha^2 x_{m,n} + \tfrac12 \beta^2 x_{m+1, n+1} + \tfrac12 \beta^2 x_{m-1, n-1}, \quad m,n \ge 1.
			\label{eq:prop:infant_mixing:3}
		\end{align}
		Fix $k \in \N$ and consider the matrix
		\begin{equation*}
			A := \begin{pmatrix} 
				2 \bigl( 1 + \frac{\mu-1}{\beta^2} \bigr) & -1
				\\
				1 & 0
			\end{pmatrix}.
		\end{equation*}
		By Equation~\eqref{eq:prop:infant_mixing:3} we have
		\begin{equation}
			\begin{pmatrix} 
				x_{m+1, m+k+1} \\ x_{m,m+k}
			\end{pmatrix}
			= A \; \begin{pmatrix}
				x_{m,m+k} \\ x_{m-1, m-1+k}
			\end{pmatrix}
			= A^m \begin{pmatrix}
				x_{1,k+1} \\ x_{0,k}
			\end{pmatrix}
			\label{eq:prop:infant_mixing:4}
		\end{equation}
		for all $m \ge 1$, $k\geq0$. Each solution of this recursion is determined by $x_{0,k}$ and $x_{1,k+1}$, that is, the set of solutions is 2-dimensional. Denote by $\omega_1, \omega_2 \in \C$ the two (possibly equal) eigenvalues of $A$ and assume $\abs{\omega_1} \le \abs{\omega_2}$. Note that
		\begin{align*}
			\omega_1 \omega_2 &= \det(A) = 1,
			&
			\tfrac12 \omega_1 + \tfrac12 \omega_2 &= \tfrac12 \Tr(A) = 1 + \tfrac{\mu-1}{\beta^2}.
		\end{align*}
		If $\abs{\omega_2} >1$ then setting $\tilde x_{m,m+k} := \omega_2^m$, $m \in \N$, provides an unbounded solution for~\eqref{eq:prop:infant_mixing:4}. If $\abs{\omega_2} \le 1$ then $\omega_2 = \omega_1^{-1} = \overline{\omega_1}$, since $\abs{\omega_1}\leq\abs{\omega_2}$ and $\omega_1 \omega_2=1$. Hence $\mu$ must be real, \ie $\mu = \pm 1$.
		For $\mu=1$ setting $\tilde x_{m,m+k} := m$, $m \in \N$, provides an unbounded solution for \eqref{eq:prop:infant_mixing:4}. For $\mu = (-1)$ the assumption $\beta^2 < 1$ leads to the contradiction
		\begin{equation*}
			-1 \le \Re \omega_1=\tfrac12 \omega_1 + \tfrac12 \bar\omega_1 = 1 + \tfrac{\mu-1}{\beta^2} < 1 - 2 = -1\,,
		\end{equation*}
		and for $\beta^2 = 1$ setting $\tilde x_{m,m+k} := (-1)^m \cdot m$, $m \in \N$, again provides an unbounded solution of \eqref{eq:prop:infant_mixing:4}.
		In any case the multiples of $\tilde x_{m,m+k} := \omega_1^m$, $m \in \N$, are the only bounded solutions of Equation~\eqref{eq:prop:infant_mixing:4}. In particular, $x_{m,m+k} = \omega_1^m \cdot x_{0,k}$ for all $m \in \N$.

		Now Equation~\eqref{eq:prop:infant_mixing:1} and Equation~\eqref{eq:prop:infant_mixing:2} simplify to
		\begin{align*}
			0 &= (\tfrac12 \alpha^2 + \tfrac12 \beta^2 \omega_1 +  \tfrac12 - \mu) \, x_{0,0} = \bigl((1-\tfrac12\beta^2)+\tfrac12\beta^2\omega_1-\mu\bigr)\, x_{0,0}\,,
			\\
			0 &= (\tfrac12 \alpha^2 + \tfrac12 \beta^2 \omega_1 + \tfrac12 \alpha - \mu)\,  x_{0,k}= \Bigl(\tfrac12\alpha+\tfrac12\bigl((1-\beta^2)+\beta^2\omega_1\bigr)-\mu\Bigr)\, x_{0,k}
		\end{align*}
		for all $k \ge 1$.
		Observe that the factor in the first (second) equation only vanishes if $\mu$ is a non-trivial convex combination of $1$ and $\omega_1$ (and $\alpha$). Since $\mu$ is an extremal point of the unit disc and $\alpha \neq \mu$, we may conclude that $x_{0,k} = 0$ for all $k \ge 1$. It follows $x_{m,m+k} = 0$ for all $m \in \N$ and $k \ge 1$. Because the element $x^*$ satisfies $T_\psi(x^*) = \bar\mu x^*$, we may likewise deduce $x_{n+k,n} = 0$ for all $n \in \N$ and $k \ge 1$, \ie, $x$~is a diagonal operator.
		For $\mu = 1$ this completes the proof; we are left with the trivial fixed point $x = x_{0,0} \cdot \one$. For $\mu \neq 1$, again using extremality of $\mu$ we may conclude that $x_{0,0} = 0$, \ie $x = 0$.
	\item
		Let $\lambda>\nicefrac12$, define a bounded diagonal operator
		\begin{align}
			d := \begin{pmatrix}
				1  \\
				& \tfrac{1-\lambda}{\lambda} \\
				&  & \left(\tfrac{1-\lambda}{\lambda}\right)^2 \\
				&&& \ddots &
			\end{pmatrix}\label{eq:prop:infant_fixedSpace_zeta=0:d}
		\end{align}
		and set $x:=\left(\tfrac{\lambda}{1-\lambda}+\tfrac{2\lambda-1}{1-\lambda}\alpha\right)\one-d$. Then a straightforward computation verifies that for each $n \ge 1$ the element $y_n:=s^n x$ is a fixed point of $T_\psi$. Moreover, the set $\{y_n\,\vert\, n\geq1\}$ is obviously linearly independent and hence the fixed space is infinite dimensional.\qedhere
	\end{enumerate}
\end{proof}

\subsection*{Baby maser}

Now we consider a further restriction of the model, setting $\alpha_n := \alpha := 0$ and $\beta_n:=\beta:=1$ for all $n\geq1$. We refer to this choice of model parameters as the \emph{baby maser}. In this case the Kraus decomposition of Proposition~\vref{prop:KrausDecomp} reduces to
\begin{gather}
	T_\psi(x) = \lambda \cdot (t_1^* x t_1 + t_2^* x t_2) \; +\;  (1-\lambda)(1-\abs{\zeta}^2) \cdot (t_3^* x t_3 + t_4^* x t_4)
	\label{eq:Kraus_Baby}
	\shortintertext{with}
	t_1 = i\zeta \sqrt{\tfrac{1-\lambda}{\lambda}} \, s^* \;,
	\qquad
	t_2 = s - i\zeta\sqrt{\tfrac{1-\lambda}{\lambda}} \, p_0 \;,
	\qquad
	t_3 = s^* \;,
	\qquad
	t_4 = p_0 
	\notag
\end{gather}
for the state $\psi$ on $M_2$ with parameters $0\le \lambda \le 1$ and $\zeta \in \D$ as in Equation~\vref{eq:density}. A~diagram of the action of $T_\psi$ on the canonical matrix units of $\BH$ is shown in Figure~\vref{fig:action_toy}. Notice that for an arbitrary state $\psi$ on $M_2$ the algebra $\ell^\infty(\N) \subseteq \BH$ of diagonal operators is invariant under~$T_\psi$.

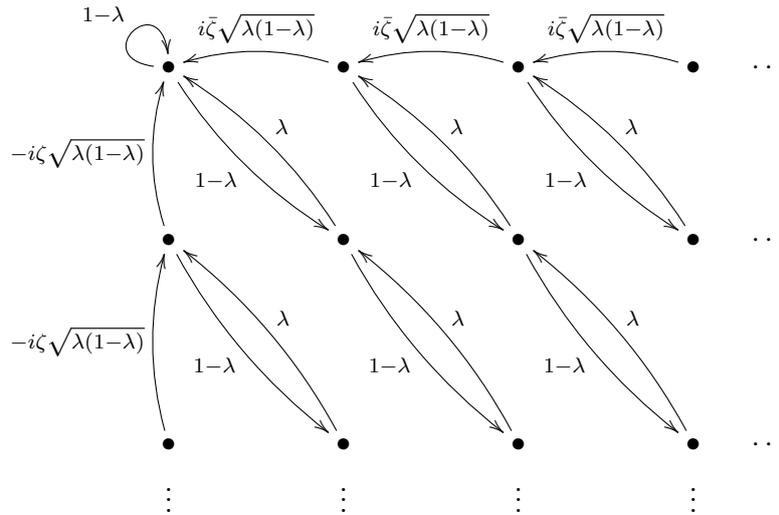
\begin{figure}[htb]
	\[
	\begin{xy}
		\xymatrix{
			\bullet \ar@/_/[rrdd]_{1-\lambda} \ar@(l,u)^{1-\lambda}
			& & \bullet \ar@/_/[rrdd]_{1-\lambda} \ar@/_/[ll]_{i\bar\zeta\sqrt{\lambda(1-\lambda)}}
			& & \bullet \ar@/_/[rrdd]_{1-\lambda} \ar@/_/[ll]_{i\bar\zeta\sqrt{\lambda(1-\lambda)}}
			& & \bullet \ar@/_/[ll]_{i\bar\zeta\sqrt{\lambda(1-\lambda)}}
			& \hspace{-1em}\cdots \\ & \\
			\bullet \ar@/_/[rrdd]_{1-\lambda} \ar@/^/[uu]^{-i\zeta\sqrt{\lambda(1-\lambda)}}
			& & \bullet \ar@/_/[rrdd]_{1-\lambda} \ar@/_/[uull]_{\lambda}
			& & \bullet \ar@/_/[rrdd]_{1-\lambda} \ar@/_/[uull]_{\lambda}
			& & \bullet \ar@/_/[uull]_{\lambda}
			& \hspace{-1em}\cdots \\ & \\
			**[d] \underset{\underset{\vdots}{ }}{\bullet} \ar@/^/[uu]^{-i\zeta\sqrt{\lambda(1-\lambda)}}
			& & **[d] \underset{\underset{\vdots}{ }}{\bullet} \ar@/_/[uull]_{\lambda}
			& & **[d] \underset{\underset{\vdots}{ }}{\bullet} \ar@/_/[uull]_{\lambda}
			& & **[d] \underset{\underset{\vdots}{ }}{\bullet} \ar@/_/[uull]_{\lambda}
			& \hspace{-1em}\cdots
		}
	\end{xy}
	\]
	\caption{Action of $T_\psi$ for $\alpha_n=0$ and $\beta_n=1$ for all $n\geq1$.}
	\label{fig:action_toy}
\end{figure}

\begin{thm}
	\label{thm:babymaser_mixing}
	Let $\alpha_n=0$ and $\beta_n=1$ for all $n\geq 1$ and let $\psi$ be an arbitrary state on~$M_2$ parametrized as in Equation \eqref{eq:density} with $0 \le \lambda \le 1$ and $\zeta\in\D$.
	\\
	Then $T_\psi$ is weakly mixing (and hence ergodic) if and only if $\lambda \le \nicefrac12$. For $\lambda > \nicefrac12$ the fixed point space is infinite dimensional.
\end{thm}
\begin{proof}
	For an element $x \in \B(\H)$ we denote by $x_{m,n} := \scal{x\,e_n,e_m}$ its coefficients. Let $x \in \B(\H)$ and $\mu \in \C$, $\abs \mu = 1$. Then $T_\psi(x) = \mu x$ if and only if the following equations are satisfied (\cf Figure~\ref{fig:action_toy}):
	\begin{align}
		\mu \, x_{0,0} &= (1-\lambda) x_{0,0} + \lambda x_{1,1} + i \bar\zeta \sqrt{\lambda(1-\lambda)} \, x_{0,1} - i \zeta \sqrt{\lambda(1-\lambda)} \, x_{1,0} \,,
		\label{eq:thm:baby_mixing:1}
		\\
		\mu \, x_{m,0} &= \lambda x_{m+1,1} - i \zeta \sqrt{\lambda(1-\lambda)} \, x_{m+1,0}\,, 
		\qquad  \qquad
		m \ge 1\,,
		\label{eq:thm:baby_mixing:2}
		\\
		\mu \, x_{0,n} &= \lambda x_{1,n+1} + i \bar\zeta \sqrt{\lambda(1-\lambda)} \, x_{0,n+1}\,,
		\qquad  \qquad
		n \ge 1\,,
		\label{eq:thm:baby_mixing:3}
		\\
		\mu \, x_{m,n} &= \lambda x_{m+1,n+1} + (1-\lambda) x_{m-1,n-1}\,,
		\qquad \qquad
		n,m \ge 1\,.
		\label{eq:thm:baby_mixing:4}
	\end{align}
	First let $0 < \lambda \le \nicefrac12$. (The case $\lambda=0$ was already studied in Example~\vref{expl:psi+-}.) Consider the matrix 
	\begin{equation*}
		A := \begin{pmatrix} 
			\frac{\mu}{\lambda} & -\frac{1-\lambda}{\lambda} \\ 
			1 & 0 
		\end{pmatrix}
	\end{equation*}
	By Equation~\ref{eq:thm:baby_mixing:4} we have
	\begin{equation}
		\begin{pmatrix} 
			x_{m+1,m+k+1} \\ x_{m,m+k}
		\end{pmatrix}
		= A \begin{pmatrix}
			x_{m,m+k} \\ x_{m-1,m+k-1}
		\end{pmatrix}
		= A^m \begin{pmatrix}
			x_{1,k+1} \\ x_{0,k}
		\end{pmatrix}
		\label{eq:thm:baby_mixing:5}
	\end{equation}
	for all $m \ge 1$, $k \ge 0$. Hence the solutions of this recursion are determined by $x_{0,k}$ and $x_{1,k+1}$, that is, the set of solutions is 2-dimensional.
	Denote by $\omega_1$, $\omega_2$ the two (possibly equal) eigenvalues of $A$ and assume $\abs{\omega_1} \le \abs{\omega_2}$. Note that 
	\begin{align*}
		\omega_1 \omega_2 &= \det(A) = \tfrac{1-\lambda}{\lambda}\,,
		& 
		\omega_1 + \omega_2 &= \Tr(A) = \lambda^{-1}\mu\,.
	\end{align*}
	For $\lambda < \nicefrac12$ we have $\omega_1 \omega_2 > 1$, and hence $\abs{\omega_2} > 1$. This provides the unbounded solution $\tilde x_{m,m+k} := \omega_2^m$ ($m \in \N$) for Equation~\eqref{eq:thm:baby_mixing:5}. For $\lambda = \nicefrac12$ we have $\tfrac12 \omega_1 + \tfrac12 \omega_2 = \mu$. Since $\mu$ is an extremal point of the unit disc, we have either $\abs{\omega_2} > 1$ or $\omega_1 = \omega_2 = \mu$. In the first case again $\tilde x_{m,m+k} := \omega_2^m$ ($m \in \N$) is an unbounded solution of \eqref{eq:thm:baby_mixing:5}. In the latter case we have $\mu^2 = \omega_1 \omega_2 = 1$ and hence $\mu = \pm 1$. Then $\tilde x_{m,m+k} := \mu^m \cdot m$ ($m \in \N$) provides an unbounded solution. In any case the space of bounded solutions of Equation~\eqref{eq:thm:baby_mixing:5} is one-dimensional. Consequently, the coefficients of $x$ are given by $x_{m,m+k} = \omega_1^m \cdot x_{0,k}$ for all $k,m \ge 0$.

	From Equation~\eqref{eq:thm:baby_mixing:3} we now deduce
	\begin{equation*}
		(\mu - \lambda \omega_1) x_{0,k} = i \bar\zeta \sqrt{\lambda(1-\lambda)} \, x_{0,k+1}\,,
	\end{equation*}
	for every $k \ge 1$. For $\zeta \neq 0$ we obtain $x_{0,k+1} = \Bigl( \frac{\mu-\lambda\omega_1}{i \bar\zeta \sqrt{\lambda(1-\lambda)}} \Bigr)^k x_{0,1}$ for every $k \in \N$. Since $(x_{0,k})_{k \in \N} = x^* e_0$ is a square summable sequence and
	\begin{equation*}
		\abs[\Big]{\frac{\mu - \lambda \omega_1}{i \bar \zeta \sqrt{\lambda(1-\lambda)}}}
		\ge \abs[\Big]{\frac{\lambda \omega_2}{\sqrt{\lambda(1-\lambda)}}} 
		= \abs[\Big]{\frac{\omega_2}{\omega_1}}^{\tfrac12} \ge 1\,,
	\end{equation*}
	it follows that $x_{0,k} = 0$ for every $k \ge 1$. For $\zeta = 0$ we obviously have $x_{0,k} = 0$ for every $k \ge 1$. Therefore, $x_{m,m+k} = 0$ for all $m \in \N$, $k \ge 1$, and, analogously, we obtain $x_{n+k,n} = 0$ for all $n \in \N$, $k \ge 1$. For $\mu = 1$ we conclude that $x \in \C \one$. For $\mu \neq 1$ Equation~\eqref{eq:thm:baby_mixing:1} reduces to
	\begin{align*}
		0 &= -\mu\, x_{0,0} + (1-\lambda) x_{0,0} + \lambda x_{1,1}
		= (1-\lambda-\mu) x_{0,0} + \lambda \omega_1 x_{0,0}
		\\
		&= (1-\lambda-\lambda\omega_2) x_{0,0}
		= (1-\lambda-\tfrac{1-\lambda}{\omega_1}) x_{0,0}
		= (1-\lambda)(1-\omega_1^{-1}) x_{0,0}\,.
	\end{align*}
	This yields $x_{0,0} = 0$ and hence $x = 0$.
	
	Conversely, let $\lambda>\nicefrac12$. Define the diagonal operator $d$ as in Equation~\vref{eq:prop:infant_fixedSpace_zeta=0:d} 
	and set $x:=\lambda (\one-s^*ds)-i\bar\zeta\sqrt{\lambda(1-\lambda)}\,(\one -d)$. Then it is a straightforward computation to check that $y_n:=s^nx$ 
	satisfies Equations \eqref{eq:thm:baby_mixing:1} through \eqref{eq:thm:baby_mixing:4} for each $n\geq 1$, that is, $y_n$ is a fixed point of $T_\psi$. Moreover, the set $\{y_n\,\vert\, n\geq1\}$ is linearly independent and hence the fixed space is infinite dimensional.
\end{proof}

\begin{thm}
	\label{thm:babymaser_invariant_state}
	Let $\alpha_n=0$ and $\beta_n=1$ for all $n\geq 1$ and let $\psi$ be an arbitrary state on~$M_2$ parametrized as in Equation \eqref{eq:density} with $0 \le \lambda \le 1$ and $\zeta\in\D$.
	\\
	Then the transition operator $T_\psi$ admits an invariant normal state $\varphi$ on $\BH$ if and only if $\lambda<\nicefrac12$. In this case $\varphi$ is given by 
	\begin{equation*}
		\overline{\varphi(e_{n+k,n})} 
		= \varphi(e_{n,n+k}) 
		= \frac{1-2\lambda}{1-\lambda} \biggl( i \bar \zeta\sqrt{\frac{\lambda}{1-\lambda}} \; \biggr)^k \biggl( \frac{\lambda}{1 - \lambda} \biggr)^n \qquad (n,k \in \N) \;.
	\end{equation*}
\end{thm}
\begin{rmk}
	Note that for $\lambda < \nicefrac12$ the invariant normal state $\varphi$ is unique, since $T_\psi$ is ergodic. Moreover, if the state $\psi$ is pure then also $\varphi$ is pure by Proposition~\ref{prop:infant_pureInvariantState}; and if $\psi$ is faithful then also $\varphi$ is faithful, since $T_\psi$ is irreducible (\cf Theorem~\vref{thm:irred_weaklymix}).
\end{rmk}

\begin{proof}
	First let $\lambda\geq \nicefrac12$ and assume that there is an invariant normal state \mbox{$\varphi=\varphi\circ T_\psi$} on~$\BH$. Then this state satisfies 
	\begin{gather*}
		\varphi(p_n) = \lambda \varphi(p_{n-1}) + (1-\lambda) \varphi(p_{n+1})\;,
		\shortintertext{or equivalently,}
		\varphi(p_n) - \varphi(p_{n-1}) = \frac{1-\lambda}{\lambda}  \bigl( \varphi(p_{n+1}) - \varphi(p_n) \bigr)
	\end{gather*}
	for every $n \ge 1$. Since $\frac{1-\lambda}{\lambda} \le 1$, the sequence $\bigl(\abs{\varphi(p_{n+1}) - \varphi(p_n)}\bigr)_{n \in \N}$ is non-decreasing. Moreover, $1 = \varphi(\one) = \sum_{n=0}^\infty \varphi(p_n)$ implies that this sequence converges to zero as $n \to \infty$. Consequently, we have $\abs{\varphi(p_{n+1}) - \varphi(p_n)} = 0$ for every $n \in \N$, \ie \mbox{$\varphi(p_{n+1}) = \varphi(p_n)$}. But this contradicts $1 = \varphi(\one) = \sum_{n=0}^\infty \varphi(p_n)$.

	Now let $\lambda<\nicefrac12$. Since $T_\psi$ is ergodic by Theorem~\ref{thm:babymaser_mixing}, due to uniqueness of the Jordan decomposition it suffices to find an invariant normal functional $\varphi \neq 0$ to prove the asserted existence. This functional necessarily satisfies $\varphi(\one) \neq 0$ and the unique invariant normal state is then given by $\varphi / \varphi(\one)$.
	We define a diagonal trace class operator
	\begin{align*}
		\tilde d := \begin{pmatrix}
			1  \\
			& \tfrac{\lambda}{1-\lambda} \\
			&  & \left(\tfrac{\lambda}{1-\lambda}\right)^2 \\
			&&& \ddots &
		\end{pmatrix}
	\end{align*}
	and set $v:=\sum_{k=1}^\infty\Bigl(i\bar\zeta\sqrt{\tfrac{\lambda}{1-\lambda}}\, s\Bigr)^k$. Then $\rho:=\frac{1-2\lambda}{1-\lambda}\bigl(\tilde d+v\tilde d+\tilde d v^*\bigr)$ is a trace class operator that satisfies the relations
	\begin{align*}
		(1-\lambda)s^*\rho s=\lambda\rho,\qquad \lambda s\rho s^*=(1-\lambda)p_0^\bot\rho p_0^\bot,\qquad i\bar\zeta\sqrt{\lambda}\, s\rho p_0=\sqrt{1-\lambda}\, p_0^\bot \rho p_0.
	\end{align*}
	A straightforward computation verifies
	\begin{align*}
		 \rho=\;& (\sqrt{\lambda}s-i\zeta\sqrt{1-\lambda}\,p_0)\rho (\sqrt{\lambda}s-i\zeta\sqrt{1-\lambda}\,p_0)^*\\
		 & +(1-\lambda)s^*\rho s+(1-\lambda)(1-\abs{\zeta}^2)p_0\rho p_0\\
		 =\;&\lambda \cdot (t_1 \rho t_1^* + t_2 \rho t_2^*) \; +\;  (1-\lambda)(1-\abs{\zeta}^2) \cdot (t_3 \rho t_3^* + t_4 \rho t_4^*).
	\end{align*}
	Hence for the functional $\varphi := \Tr(\rho \,\cdot\,)$ we obtain $\varphi \circ T_\psi = \varphi$ using the trace property. Clearly, $\varphi$ is of the asserted form and satisfies $\varphi(\one) = 1$.
\end{proof}

\section{Summary}	\label{sec:summary}

In the previous sections we studied our class of quantum birth and death chains for different choices of model parameters $\alpha_n$, $\beta_n$. As a summary we regroup our results by the examined properties and visualize them. Throughout this section $\psi$ denotes a state on $M_2$ parametrized as in Equation~\eqref{eq:density} with parameters $0 \le \lambda \le 1$ and $\zeta \in \mathbb D$.

\subsection*{Irreducibility}
For general model parameters the transition operator $T_\psi$ is irreducible if $\psi$ is faithful, \ie, if $0< \lambda < 1$ and $\abs\zeta < 1$. It is not irreducible for the pure states $\psi_+$ and $\psi_-$. Moreover, for the model parameters $\alpha_n = \alpha$ and $\beta_n = \beta$, $n \ge 1$, it is not irreducible for pure states $\psi$ with $\lambda < \tfrac12(1-\alpha)$. 
These results can be deduced from Thm.\,\ref{thm:irred_weaklymix}, Ex.\,\ref{expl:psi+-}, and Prop.\,\ref{prop:infant_pureInvariantState}. 
For pure states $\psi$ with $\tfrac12(1-\alpha) < \lambda < 1$ we do not have any results concerning irreducibility.

\begin{figure}[htb]
	\centering
	\begin{minipage}[c]{0.3\linewidth}
		\centering
		\begin{small}\begin{tikzpicture}[line cap=round,line join=round,>=triangle 45,x=1.8cm,y=1.8cm]
			\fill[fill=red,fill opacity=0.4] (1.9,3) -- (1.9,2.9) -- (2,2.9) -- (2,3) -- cycle; 
			\fill[fill=red,fill opacity=0.4] (1.9,1) -- (1.9,1.1) -- (2,1.1) -- (2,1) -- cycle; 
			\fill[fill=blue,fill opacity=0.6] (2,2.9) -- (1.9,2.9) -- (1.9,1.1) -- (2,1.1) -- cycle; 
			\fill [shift={(2,2)},fill=blue,fill opacity=0.6]  (0,0) --  plot[domain=-1.57:1.57,variable=\t]({1*0.9*cos(\t r)+0*0.9*sin(\t r)},{0*0.9*cos(\t r)+1*0.9*sin(\t r)}) -- cycle; 
			\draw [line width=1pt] (1.9,3)-- (2,3); 
			\draw [line width=1pt] (1.9,2.9)-- (2,2.9); 
			\draw [line width=1pt] (1.9,1.1)-- (2,1.1); 
			\draw [line width=1pt] (1.9,1)-- (2,1); 
			\draw [line width=1pt] (1.9,3)-- (1.9,1); 
			\draw [line width=1pt] (2,1)-- (2,3); 
			\draw [shift={(2,2)},line width=1pt]  plot[domain=-1.57:1.57,variable=\t]({1*1*cos(\t r)+0*1*sin(\t r)},{0*1*cos(\t r)+1*1*sin(\t r)}); 
			\draw [shift={(2,2)},line width=1pt]  plot[domain=-1.57:1.57,variable=\t]({1*0.9*cos(\t r)+0*0.9*sin(\t r)},{0*0.9*cos(\t r)+1*0.9*sin(\t r)}); 

			\draw (1.7,2.95) node[anchor=east] {$\lambda=1$};
			\draw [dash pattern=on .5pt off 1.5pt,line width=.5pt] (1.8,2.95)-- (1.7,2.95);
			\draw (1.7,1.05) node[anchor=east] {$\lambda=0$};
			\draw [dash pattern=on .5pt off 1.5pt,line width=.5pt] (1.8,1.05)-- (1.7,1.05);
			\draw (2.62,2.7) node[anchor=south west] {$|\zeta|=1$};
			\draw (1.97,2.99) node[anchor=south] {$\psi_+$};
			\draw (1.97,1.01) node[anchor=north] {$\psi_-$};
		\end{tikzpicture}\\
		\emph{Arbitrary model\\ parameters $\alpha_n,\beta_n$.}\end{small}
	\end{minipage}\hspace{0.04\linewidth}
	\begin{minipage}[c]{0.3\linewidth}
		\centering
		\begin{small}\begin{tikzpicture}[line cap=round,line join=round,>=triangle 45,x=1.8cm,y=1.8cm]
			\fill[fill=red,fill opacity=0.4] (1.9,3) -- (1.9,2.9) -- (2,2.9) -- (2,3) -- cycle; 
			\fill[fill=red,fill opacity=0.4] (1.9,1) -- (1.9,1.1) -- (2,1.1) -- (2,1) -- cycle; 
			\fill[fill=blue,fill opacity=0.6] (2,2.9) -- (1.9,2.9) -- (1.9,1.1) -- (2,1.1) -- cycle; 
			\fill [shift={(2,2)},fill=blue,fill opacity=0.6]  (0,0) --  plot[domain=-1.57:1.57,variable=\t]({1*0.9*cos(\t r)+0*0.9*sin(\t r)},{0*0.9*cos(\t r)+1*0.9*sin(\t r)}) -- cycle; 
			\fill [shift={(2,2)},fill=red,fill opacity=0.4]  plot[domain=-0.34:-1.57,variable=\t]({1*0.9*cos(\t r)+0*0.9*sin(\t r)},{0*0.9*cos(\t r)+1*0.9*sin(\t r)}) -- plot[domain=-1.57:-0.3,variable=\t]({1*1*cos(\t r)+0*1*sin(\t r)},{0*1*cos(\t r)+1*1*sin(\t r)}) -- cycle; 

			\draw [line width=1pt] (1.9,1.7)-- (2.95,1.7); 
			\draw [line width=1pt] (1.9,3)-- (2,3); 
			\draw [line width=1pt] (1.9,2.9)-- (2,2.9); 
			\draw [line width=1pt] (1.9,1.1)-- (2,1.1); 
			\draw [line width=1pt] (1.9,1)-- (2,1); 
			\draw [line width=1pt] (1.9,3)-- (1.9,1); 
			\draw [line width=1pt] (2,1)-- (2,3); 
			\draw [shift={(2,2)},line width=1pt]  plot[domain=-1.57:1.57,variable=\t]({1*1*cos(\t r)+0*1*sin(\t r)},{0*1*cos(\t r)+1*1*sin(\t r)}); 
			\draw [shift={(2,2)},line width=1pt]  plot[domain=-1.57:1.57,variable=\t]({1*0.9*cos(\t r)+0*0.9*sin(\t r)},{0*0.9*cos(\t r)+1*0.9*sin(\t r)}); 

			\draw (1.7,2.95) node[anchor=east] {$\lambda=1$};
			\draw [dash pattern=on .5pt off 1.5pt,line width=.5pt] (1.8,2.95)-- (1.7,2.95);
			\draw (1.7,1.7) node[anchor=east] {$\frac12(1\!-\!\alpha)$};
			\draw [dash pattern=on .5pt off 1.5pt,line width=.5pt] (1.8,1.7)-- (1.7,1.7);
			\draw (1.7,1.05) node[anchor=east] {$\lambda=0$};
			\draw [dash pattern=on .5pt off 1.5pt,line width=.5pt] (1.8,1.05)-- (1.7,1.05);
			\draw (2.62,2.7) node[anchor=south west] {$|\zeta|=1$};
			\draw (1.97,2.99) node[anchor=south] {$\psi_+$};
			\draw (1.97,1.01) node[anchor=north] {$\psi_-$};
		\end{tikzpicture}\\
		\emph{Toy example with\\ $\alpha_n:=\alpha$, $\beta_n:=\beta\neq0$.}\end{small}
	\end{minipage}\hspace{0.04\linewidth}
	\begin{minipage}[c]{0.3\linewidth}
		\centering
		\begin{small}\begin{tikzpicture}[line cap=round,line join=round,>=triangle 45,x=1.8cm,y=1.8cm]
			\fill[fill=red,fill opacity=0.4] (1.9,3) -- (1.9,2.9) -- (2,2.9) -- (2,3) -- cycle; 
			\fill[fill=red,fill opacity=0.4] (1.9,1) -- (1.9,1.1) -- (2,1.1) -- (2,1) -- cycle; 
			\fill[fill=blue,fill opacity=0.6] (2,2.9) -- (1.9,2.9) -- (1.9,1.1) -- (2,1.1) -- cycle; 
			\fill [shift={(2,2)},fill=blue,fill opacity=0.6]  (0,0) --  plot[domain=-1.57:1.57,variable=\t]({1*0.9*cos(\t r)+0*0.9*sin(\t r)},{0*0.9*cos(\t r)+1*0.9*sin(\t r)}) -- cycle; 
			\fill [shift={(2,2)},fill=red,fill opacity=0.4]  plot[domain=-0.05:-1.57,variable=\t]({1*0.9*cos(\t r)+0*0.9*sin(\t r)},{0*0.9*cos(\t r)+1*0.9*sin(\t r)}) --  plot[domain=-1.57:-0.05,variable=\t]({1*1*cos(\t r)+0*1*sin(\t r)},{0*1*cos(\t r)+1*1*sin(\t r)}) -- cycle; 

			\draw [line width=1pt] (1.9,2.05)-- (2.99,2.05); 
			\draw [line width=1pt] (1.9,1.95)-- (2.99,1.95); 
			\draw [line width=1pt] (1.9,3)-- (2,3); 
			\draw [line width=1pt] (1.9,2.9)-- (2,2.9); 
			\draw [line width=1pt] (1.9,1.1)-- (2,1.1); 
			\draw [line width=1pt] (1.9,1)-- (2,1); 
			\draw [line width=1pt] (1.9,3)-- (1.9,1); 
			\draw [line width=1pt] (2,1)-- (2,3); 
			\draw [shift={(2,2)},line width=1pt]  plot[domain=-1.57:1.57,variable=\t]({1*1*cos(\t r)+0*1*sin(\t r)},{0*1*cos(\t r)+1*1*sin(\t r)}); 
			\draw [shift={(2,2)},line width=1pt]  plot[domain=-1.57:1.57,variable=\t]({1*0.9*cos(\t r)+0*0.9*sin(\t r)},{0*0.9*cos(\t r)+1*0.9*sin(\t r)}); 

			\draw (1.7,2.95) node[anchor=east] {$\lambda=1$};
			\draw [dash pattern=on .5pt off 1.5pt,line width=.5pt] (1.8,2.95)-- (1.7,2.95);
			\draw (1.7,2) node[anchor=east] {$\lambda=\frac12$};
			\draw [dash pattern=on .5pt off 1.5pt,line width=.5pt] (1.8,2)-- (1.7,2);
			\draw (1.7,1.05) node[anchor=east] {$\lambda=0$};
			\draw [dash pattern=on .5pt off 1.5pt,line width=.5pt] (1.8,1.05)-- (1.7,1.05);
			\draw (2.62,2.7) node[anchor=south west] {$|\zeta|=1$};
			\draw (1.97,2.99) node[anchor=south] {$\psi_+$};
			\draw (1.97,1.01) node[anchor=north] {$\psi_-$};
		\end{tikzpicture}\\
		\emph{Baby maser, \ie,\\ $\alpha_n:=0$, $\beta_n:=1$.}\end{small}
	\end{minipage}
	\caption[Irreducibility]{These figures show the parameter regions\footnotemark{} for which the transition operator~$T_\psi$ with the specified choices of model parameters is \emph{irreducible}. Blue areas indicate that $T_\psi$ is irreducible, red areas indicate that $T_\psi$ is not irreducible, and regions for which we do not have any results are left blank.}
\end{figure}
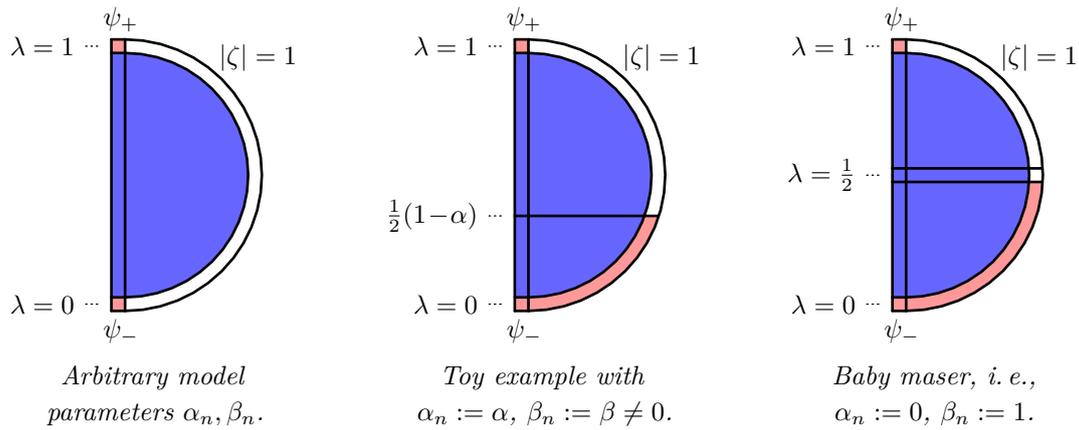\vspace{2ex}

\subsection*{Invariant Normal States}
\footnotetext{By Proposition \ref{prop:rotation} and Remark \vref{rmk:rotation} our results do not depend on the phase of $\zeta$. To visualize our results it is therefore enough to draw the right half of a cross section of the Bloch\;\textquoteleft sphere\textquoteright.}

For general model parameters and diagonal $\psi$ ($\zeta = 0$) the transition operator $T_\psi$ admits an invariant normal state if and only if  $\lambda < \nicefrac12$. For non-diagonal states $\psi$ we do not have general results about invariant normal states.

For $\alpha_n = \alpha$ and $\beta_n = \beta$, $n \ge 1$, we concentrated on studying pure states $\psi$ on $M_2$. For a pure state $\psi$ with $\lambda < \tfrac12(1-\alpha)$ we computed a pure invariant normal state for $T_\psi$. For pure states with $\lambda > \tfrac12(1-\alpha)$ we do not know whether there is an invariant normal state, but we showed that, if it exists, it cannot be pure.

For the baby maser ($\alpha_n = 0$, $\beta_n =1$) we found a complete characterization: For an arbitrary state $\psi$ the transition operator admits an invariant normal state if and only if $\lambda < \tfrac12$. We computed this state explicitly.

We showed for arbitrary model parameters that for a faithful state $\psi$ the invariant normal state is unique and faithful (given it exists at all). We do not know whether for a pure $\psi$ the invariant normal state must be pure, too. However, in all cases where we computed the invariant state this turned out to be true.
For these results consult Thm.\,\ref{thm:irred_weaklymix}, Prop.\,\ref{prop:invariantState_zeta=0} and \ref{prop:infant_pureInvariantState}, and Thm.\,\ref{thm:babymaser_invariant_state}.

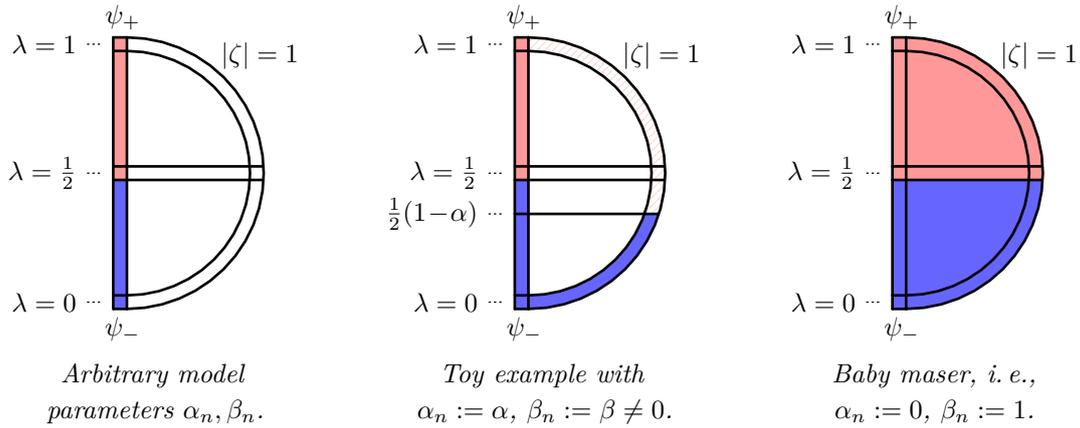
\begin{figure}[htb]
	\centering
	\begin{minipage}[c]{0.3\linewidth}
		\centering
		\begin{small}\begin{tikzpicture}[line cap=round,line join=round,>=triangle 45,x=1.8cm,y=1.8cm]
			\fill[fill=red,fill opacity=0.4] (1.9,3) -- (1.9,2.9) -- (2,2.9) -- (2,3) -- cycle; 
			\fill[fill=blue,fill opacity=0.6] (1.9,1) -- (1.9,1.1) -- (2,1.1) -- (2,1) -- cycle; 
			\fill[fill=red,fill opacity=0.4] (2,1.95) -- (1.9,1.95) -- (1.9,2.9) -- (2,2.9) -- cycle; 
			\fill[fill=blue,fill opacity=0.6] (2,1.95) -- (1.9,1.95) -- (1.9,1.1) -- (2,1.1) -- cycle; 

			\draw [line width=1pt] (1.9,2.05)-- (2.99,2.05); 
			\draw [line width=1pt] (1.9,1.95)-- (2.99,1.95); 
			\draw [line width=1pt] (1.9,3)-- (2,3); 
			\draw [line width=1pt] (1.9,2.9)-- (2,2.9); 
			\draw [line width=1pt] (1.9,1.1)-- (2,1.1); 
			\draw [line width=1pt] (1.9,1)-- (2,1); 
			\draw [line width=1pt] (1.9,3)-- (1.9,1); 
			\draw [line width=1pt] (2,1)-- (2,3); 
			\draw [shift={(2,2)},line width=1pt]  plot[domain=-1.57:1.57,variable=\t]({1*1*cos(\t r)+0*1*sin(\t r)},{0*1*cos(\t r)+1*1*sin(\t r)}); 
			\draw [shift={(2,2)},line width=1pt]  plot[domain=-1.57:1.57,variable=\t]({1*0.9*cos(\t r)+0*0.9*sin(\t r)},{0*0.9*cos(\t r)+1*0.9*sin(\t r)}); 

			\draw (1.7,2.95) node[anchor=east] {$\lambda=1$};
			\draw [dash pattern=on .5pt off 1.5pt,line width=.5pt] (1.8,2.95)-- (1.7,2.95);
			\draw (1.7,2) node[anchor=east] {$\lambda=\frac12$};
			\draw [dash pattern=on .5pt off 1.5pt,line width=.5pt] (1.8,2)-- (1.7,2);
			\draw (1.7,1.05) node[anchor=east] {$\lambda=0$};
			\draw [dash pattern=on .5pt off 1.5pt,line width=.5pt] (1.8,1.05)-- (1.7,1.05);
			\draw (2.62,2.7) node[anchor=south west] {$|\zeta|=1$};
			\draw (1.97,2.99) node[anchor=south] {$\psi_+$};
			\draw (1.97,1.01) node[anchor=north] {$\psi_-$};
		\end{tikzpicture}\\
		\emph{Arbitrary model\\ parameters $\alpha_n,\beta_n$.}\end{small}
	\end{minipage}\hspace{0.04\linewidth}
	\begin{minipage}[c]{0.3\linewidth}
		\centering
		\begin{small}\begin{tikzpicture}[line cap=round,line join=round,>=triangle 45,x=1.8cm,y=1.8cm]
			\fill[fill=red,fill opacity=0.4] (1.9,3) -- (1.9,2.9) -- (2,2.9) -- (2,3) -- cycle; 
			\fill[fill=blue,fill opacity=0.6] (1.9,1) -- (1.9,1.1) -- (2,1.1) -- (2,1) -- cycle; 
			\fill[fill=red,fill opacity=0.4] (2,1.95) -- (1.9,1.95) -- (1.9,2.9) -- (2,2.9) -- cycle; 
			\fill[fill=blue,fill opacity=0.6] (2,1.95) -- (1.9,1.95) -- (1.9,1.1) -- (2,1.1) -- cycle; 
			\fill [shift={(2,2)},pattern color=red,pattern=north east lines,opacity=0.4]  plot[domain=1.57:-0.34,variable=\t]({1*0.9*cos(\t r)+0*0.9*sin(\t r)},{0*0.9*cos(\t r)+1*0.9*sin(\t r)}) --  plot[domain=-0.3:1.57,variable=\t]({1*1*cos(\t r)+0*1*sin(\t r)},{0*1*cos(\t r)+1*1*sin(\t r)}) -- cycle; 
			\fill [shift={(2,2)},fill=blue,fill opacity=0.6]  plot[domain=-0.34:-1.57,variable=\t]({1*0.9*cos(\t r)+0*0.9*sin(\t r)},{0*0.9*cos(\t r)+1*0.9*sin(\t r)}) -- plot[domain=-1.57:-0.3,variable=\t]({1*1*cos(\t r)+0*1*sin(\t r)},{0*1*cos(\t r)+1*1*sin(\t r)}) -- cycle; 

			\draw [line width=1pt] (1.9,2.05)-- (2.99,2.05); 
			\draw [line width=1pt] (1.9,1.95)-- (2.99,1.95); 
			\draw [line width=1pt] (1.9,1.7)-- (2.95,1.7); 
			\draw [line width=1pt] (1.9,3)-- (2,3); 
			\draw [line width=1pt] (1.9,2.9)-- (2,2.9); 
			\draw [line width=1pt] (1.9,1.1)-- (2,1.1); 
			\draw [line width=1pt] (1.9,1)-- (2,1); 
			\draw [line width=1pt] (1.9,3)-- (1.9,1); 
			\draw [line width=1pt] (2,1)-- (2,3); 
			\draw [shift={(2,2)},line width=1pt]  plot[domain=-1.57:1.57,variable=\t]({1*1*cos(\t r)+0*1*sin(\t r)},{0*1*cos(\t r)+1*1*sin(\t r)}); 
			\draw [shift={(2,2)},line width=1pt]  plot[domain=-1.57:1.57,variable=\t]({1*0.9*cos(\t r)+0*0.9*sin(\t r)},{0*0.9*cos(\t r)+1*0.9*sin(\t r)}); 

			\draw (1.7,2.95) node[anchor=east] {$\lambda=1$};
			\draw [dash pattern=on .5pt off 1.5pt,line width=.5pt] (1.8,2.95)-- (1.7,2.95);
			\draw (1.7,2) node[anchor=east] {$\lambda=\frac12$};
			\draw [dash pattern=on .5pt off 1.5pt,line width=.5pt] (1.8,2)-- (1.7,2);
			\draw (1.7,1.7) node[anchor=east] {$\frac12(1\!-\!\alpha)$};
			\draw [dash pattern=on .5pt off 1.5pt,line width=.5pt] (1.8,1.7)-- (1.7,1.7);
			\draw (1.7,1.05) node[anchor=east] {$\lambda=0$};
			\draw [dash pattern=on .5pt off 1.5pt,line width=.5pt] (1.8,1.05)-- (1.7,1.05);
			\draw (2.62,2.7) node[anchor=south west] {$|\zeta|=1$};
			\draw (1.97,2.99) node[anchor=south] {$\psi_+$};
			\draw (1.97,1.01) node[anchor=north] {$\psi_-$};
		\end{tikzpicture}\\
		\emph{Toy example with\\ $\alpha_n:=\alpha$, $\beta_n:=\beta\neq0$.}\end{small}
	\end{minipage}\hspace{0.04\linewidth}
	\begin{minipage}[c]{0.3\linewidth}
		\centering
		\begin{small}\begin{tikzpicture}[line cap=round,line join=round,>=triangle 45,x=1.8cm,y=1.8cm]
			\fill[fill=red,fill opacity=0.4] (1.9,3) -- (1.9,2.9) -- (2,2.9) -- (2,3) -- cycle; 
			\fill[fill=blue,fill opacity=0.6] (1.9,1) -- (1.9,1.1) -- (2,1.1) -- (2,1) -- cycle; 
			\fill[fill=red,fill opacity=0.4] (2,1.95) -- (1.9,1.95) -- (1.9,2.9) -- (2,2.9) -- cycle; 
			\fill[fill=blue,fill opacity=0.6] (2,1.95) -- (1.9,1.95) -- (1.9,1.1) -- (2,1.1) -- cycle; 
			\fill [shift={(2,2)},fill=red,fill opacity=0.4]  (0,-0.05) --  plot[domain=-0.05:1.57,variable=\t]({1*0.9*cos(\t r)+0*0.9*sin(\t r)},{0*0.9*cos(\t r)+1*0.9*sin(\t r)}) -- cycle; 
			\fill [shift={(2,2)},fill=blue,fill opacity=0.6]  (0,-0.05) --  plot[domain=-1.57:-0.05,variable=\t]({1*0.9*cos(\t r)+0*0.9*sin(\t r)},{0*0.9*cos(\t r)+1*0.9*sin(\t r)}) -- cycle; 
			\fill [shift={(2,2)},fill=red,fill opacity=0.4]  plot[domain=1.57:-0.05,variable=\t]({1*0.9*cos(\t r)+0*0.9*sin(\t r)},{0*0.9*cos(\t r)+1*0.9*sin(\t r)}) --  plot[domain=-0.05:1.57,variable=\t]({1*1*cos(\t r)+0*1*sin(\t r)},{0*1*cos(\t r)+1*1*sin(\t r)}) -- cycle; 
			\fill [shift={(2,2)},fill=blue,fill opacity=0.6]  plot[domain=-0.05:-1.57,variable=\t]({1*0.9*cos(\t r)+0*0.9*sin(\t r)},{0*0.9*cos(\t r)+1*0.9*sin(\t r)}) --  plot[domain=-1.57:-0.05,variable=\t]({1*1*cos(\t r)+0*1*sin(\t r)},{0*1*cos(\t r)+1*1*sin(\t r)}) -- cycle; 

			\draw [line width=1pt] (1.9,2.05)-- (2.99,2.05); 
			\draw [line width=1pt] (1.9,1.95)-- (2.99,1.95); 
			\draw [line width=1pt] (1.9,3)-- (2,3); 
			\draw [line width=1pt] (1.9,2.9)-- (2,2.9); 
			\draw [line width=1pt] (1.9,1.1)-- (2,1.1); 
			\draw [line width=1pt] (1.9,1)-- (2,1); 
			\draw [line width=1pt] (1.9,3)-- (1.9,1); 
			\draw [line width=1pt] (2,1)-- (2,3); 
			\draw [shift={(2,2)},line width=1pt]  plot[domain=-1.57:1.57,variable=\t]({1*1*cos(\t r)+0*1*sin(\t r)},{0*1*cos(\t r)+1*1*sin(\t r)}); 
			\draw [shift={(2,2)},line width=1pt]  plot[domain=-1.57:1.57,variable=\t]({1*0.9*cos(\t r)+0*0.9*sin(\t r)},{0*0.9*cos(\t r)+1*0.9*sin(\t r)}); 

			\draw (1.7,2.95) node[anchor=east] {$\lambda=1$};
			\draw [dash pattern=on .5pt off 1.5pt,line width=.5pt] (1.8,2.95)-- (1.7,2.95);
			\draw (1.7,2) node[anchor=east] {$\lambda=\frac12$};
			\draw [dash pattern=on .5pt off 1.5pt,line width=.5pt] (1.8,2)-- (1.7,2);
			\draw (1.7,1.05) node[anchor=east] {$\lambda=0$};
			\draw [dash pattern=on .5pt off 1.5pt,line width=.5pt] (1.8,1.05)-- (1.7,1.05);
			\draw (2.62,2.7) node[anchor=south west] {$|\zeta|=1$};
			\draw (1.97,2.99) node[anchor=south] {$\psi_+$};
			\draw (1.97,1.01) node[anchor=north] {$\psi_-$};
		\end{tikzpicture}\\
		\emph{Baby maser, \ie,\\ $\alpha_n:=0$, $\beta_n:=1$.}\end{small}
	\end{minipage}
	\caption[Invariant normal states]{These figures show the parameter regions\footnotemark[\value{footnote}]{} for which $T_\psi$ with the specified choices of model parameters admits an \emph{invariant normal state}. Blue areas indicate that there is an invariant normal state for $T_\psi$, red areas indicate that there is no such state. For areas with red lines no \emph{pure} invariant normal state exists, and regions for which we do not have any results are left blank.
	}
\end{figure}\vspace{2ex}

\subsection*{Ergodicity and Mixing}
To study ergodicity we concentrated on diagonal states $\psi$ ($\zeta =0$). For general model parameters the transition operator $T_\psi$ of a diagonal state $\psi$ is weakly mixing (and hence ergodic) if $\lambda < \nicefrac12$. For $\lambda \ge \nicefrac12$ we do not know whether $T_\psi$ is ergodic or weakly mixing.
 
For $\alpha_n = \alpha$ and $\beta_n = \beta$, $n \ge1$, the transition operator $T_\psi$ of a diagonal state $\psi$ is also weakly mixing if $\lambda = \nicefrac12$. If $\lambda > \nicefrac12$ it is not even ergodic (and hence not weakly mixing). We provided explicitly infinitely many linearly independent fixed points.
 
For the baby maser ($\alpha_n=0$, $\beta_n=1$) we found a complete characterization for arbitrary states. Here the transition operator $T_\psi$ is weakly mixing for an arbitrary state $\psi$ in the closed lower hemisphere, \ie for $\lambda \le \nicefrac12$. In the open upper hemisphere, \ie for $\lambda > \nicefrac12$, $T_\psi$ is not ergodic and we provided infinitely many linearly independent fixed points.
These results are deduced from Thm.\,\ref{thm:irred_weaklymix}, Prop.\,\ref{prop:invariantState_zeta=0}, Thm.\,\ref{thm:babymaser_mixing}.

\begin{figure}[htb]
	\centering
	\begin{minipage}[c]{0.3\linewidth}
		\centering
		\begin{small}\begin{tikzpicture}[line cap=round,line join=round,>=triangle 45,x=1.8cm,y=1.8cm]
			\fill[pattern color=red,pattern=north east lines,opacity=0.4] (1.9,3) -- (1.9,2.9) -- (2,2.9) -- (2,3) -- cycle; 
			\fill[pattern color=blue,pattern=north west lines,opacity=0.6] (1.9,3) -- (1.9,2.9) -- (2,2.9) -- (2,3) -- cycle; 
			\fill[fill=blue,fill opacity=0.6] (1.9,1) -- (1.9,1.1) -- (2,1.1) -- (2,1) -- cycle; 
			\fill[fill=blue,fill opacity=0.6] (2,1.95) -- (1.9,1.95) -- (1.9,1.1) -- (2,1.1) -- cycle; 

			\draw [line width=1pt] (1.9,2.05)-- (2.99,2.05); 
			\draw [line width=1pt] (1.9,1.95)-- (2.99,1.95); 
			\draw [line width=1pt] (1.9,3)-- (2,3); 
			\draw [line width=1pt] (1.9,2.9)-- (2,2.9); 
			\draw [line width=1pt] (1.9,1.1)-- (2,1.1); 
			\draw [line width=1pt] (1.9,1)-- (2,1); 
			\draw [line width=1pt] (1.9,3)-- (1.9,1); 
			\draw [line width=1pt] (2,1)-- (2,3); 
			\draw [shift={(2,2)},line width=1pt]  plot[domain=-1.57:1.57,variable=\t]({1*1*cos(\t r)+0*1*sin(\t r)},{0*1*cos(\t r)+1*1*sin(\t r)}); 
			\draw [shift={(2,2)},line width=1pt]  plot[domain=-1.57:1.57,variable=\t]({1*0.9*cos(\t r)+0*0.9*sin(\t r)},{0*0.9*cos(\t r)+1*0.9*sin(\t r)}); 

			\draw (1.7,2.95) node[anchor=east] {$\lambda=1$};
			\draw [dash pattern=on .5pt off 1.5pt,line width=.5pt] (1.8,2.95)-- (1.7,2.95);
			\draw (1.7,2) node[anchor=east] {$\lambda=\frac12$};
			\draw [dash pattern=on .5pt off 1.5pt,line width=.5pt] (1.8,2)-- (1.7,2);
			\draw (1.7,1.05) node[anchor=east] {$\lambda=0$};
			\draw [dash pattern=on .5pt off 1.5pt,line width=.5pt] (1.8,1.05)-- (1.7,1.05);
			\draw (2.62,2.7) node[anchor=south west] {$|\zeta|=1$};
			\draw (1.97,2.99) node[anchor=south] {$\psi_+$};
			\draw (1.97,1.01) node[anchor=north] {$\psi_-$};
		\end{tikzpicture}\\
		\emph{Arbitrary model\\ parameters $\alpha_n,\beta_n$.}\end{small}
	\end{minipage}\hspace{0.04\linewidth}
	\begin{minipage}[c]{0.3\linewidth}
		\centering
		\begin{small}\begin{tikzpicture}[line cap=round,line join=round,>=triangle 45,x=1.8cm,y=1.8cm]
			\fill[fill=red,fill opacity=0.4] (1.9,3) -- (1.9,2.9) -- (2,2.9) -- (2,3) -- cycle; 
			\fill[fill=blue,fill opacity=0.6] (1.9,1) -- (1.9,1.1) -- (2,1.1) -- (2,1) -- cycle; 
			\fill[fill=red,fill opacity=0.4] (2,2.05) -- (1.9,2.05) -- (1.9,2.9) -- (2,2.9) -- cycle; 
			\fill[fill=blue,fill opacity=0.6] (2,2.05) -- (1.9,2.05) -- (1.9,1.1) -- (2,1.1) -- cycle; 

			\draw [line width=1pt] (1.9,2.05)-- (2.99,2.05); 
			\draw [line width=1pt] (1.9,1.95)-- (2.99,1.95); 
			\draw [line width=1pt] (1.9,3)-- (2,3); 
			\draw [line width=1pt] (1.9,2.9)-- (2,2.9); 
			\draw [line width=1pt] (1.9,1.1)-- (2,1.1); 
			\draw [line width=1pt] (1.9,1)-- (2,1); 
			\draw [line width=1pt] (1.9,3)-- (1.9,1); 
			\draw [line width=1pt] (2,1)-- (2,3); 
			\draw [shift={(2,2)},line width=1pt]  plot[domain=-1.57:1.57,variable=\t]({1*1*cos(\t r)+0*1*sin(\t r)},{0*1*cos(\t r)+1*1*sin(\t r)}); 
			\draw [shift={(2,2)},line width=1pt]  plot[domain=-1.57:1.57,variable=\t]({1*0.9*cos(\t r)+0*0.9*sin(\t r)},{0*0.9*cos(\t r)+1*0.9*sin(\t r)}); 

			\draw (1.7,2.95) node[anchor=east] {$\lambda=1$};
			\draw [dash pattern=on .5pt off 1.5pt,line width=.5pt] (1.8,2.95)-- (1.7,2.95);
			\draw (1.7,2) node[anchor=east] {$\lambda=\frac12$};
			\draw [dash pattern=on .5pt off 1.5pt,line width=.5pt] (1.8,2)-- (1.7,2);
			\draw (1.7,1.05) node[anchor=east] {$\lambda=0$};
			\draw [dash pattern=on .5pt off 1.5pt,line width=.5pt] (1.8,1.05)-- (1.7,1.05);
			\draw (2.62,2.7) node[anchor=south west] {$|\zeta|=1$};
			\draw (1.97,2.99) node[anchor=south] {$\psi_+$};
			\draw (1.97,1.01) node[anchor=north] {$\psi_-$};
		\end{tikzpicture}\\
		\emph{Toy example with\\ $\alpha_n:=\alpha$, $\beta_n:=\beta\neq0$.}\end{small}
	\end{minipage}\hspace{0.04\linewidth}
	\begin{minipage}[c]{0.3\linewidth}
		\centering
		\begin{small}\begin{tikzpicture}[line cap=round,line join=round,>=triangle 45,x=1.8cm,y=1.8cm]
			\fill[fill=red,fill opacity=0.4] (1.9,3) -- (1.9,2.9) -- (2,2.9) -- (2,3) -- cycle; 
			\fill[fill=blue,fill opacity=0.6] (1.9,1) -- (1.9,1.1) -- (2,1.1) -- (2,1) -- cycle; 
			\fill[fill=red,fill opacity=0.4] (2,2.05) -- (1.9,2.05) -- (1.9,2.9) -- (2,2.9) -- cycle; 
			\fill[fill=blue,fill opacity=0.6] (2,2.05) -- (1.9,2.05) -- (1.9,1.1) -- (2,1.1) -- cycle; 
			\fill [shift={(2,2)},fill=red,fill opacity=0.4]  (0,0.05) --  plot[domain=0.05:1.57,variable=\t]({1*0.9*cos(\t r)+0*0.9*sin(\t r)},{0*0.9*cos(\t r)+1*0.9*sin(\t r)}) -- cycle; 
			\fill [shift={(2,2)},fill=blue,fill opacity=0.6]  (0,0.05) --  plot[domain=-1.57:0.05,variable=\t]({1*0.9*cos(\t r)+0*0.9*sin(\t r)},{0*0.9*cos(\t r)+1*0.9*sin(\t r)}) -- cycle; 
			\fill [shift={(2,2)},fill=red,fill opacity=0.4]  plot[domain=1.57:0.05,variable=\t]({1*0.9*cos(\t r)+0*0.9*sin(\t r)},{0*0.9*cos(\t r)+1*0.9*sin(\t r)}) --  plot[domain=0.05:1.57,variable=\t]({1*1*cos(\t r)+0*1*sin(\t r)},{0*1*cos(\t r)+1*1*sin(\t r)}) -- cycle; 
			\fill [shift={(2,2)},fill=blue,fill opacity=0.6]  plot[domain=0.05:-1.57,variable=\t]({1*0.9*cos(\t r)+0*0.9*sin(\t r)},{0*0.9*cos(\t r)+1*0.9*sin(\t r)}) --  plot[domain=-1.57:0.05,variable=\t]({1*1*cos(\t r)+0*1*sin(\t r)},{0*1*cos(\t r)+1*1*sin(\t r)}) -- cycle; 

			\draw [line width=1pt] (1.9,2.05)-- (2.99,2.05); 
			\draw [line width=1pt] (1.9,1.95)-- (2.99,1.95); 
			\draw [line width=1pt] (1.9,3)-- (2,3); 
			\draw [line width=1pt] (1.9,2.9)-- (2,2.9); 
			\draw [line width=1pt] (1.9,1.1)-- (2,1.1); 
			\draw [line width=1pt] (1.9,1)-- (2,1); 
			\draw [line width=1pt] (1.9,3)-- (1.9,1); 
			\draw [line width=1pt] (2,1)-- (2,3); 
			\draw [shift={(2,2)},line width=1pt]  plot[domain=-1.57:1.57,variable=\t]({1*1*cos(\t r)+0*1*sin(\t r)},{0*1*cos(\t r)+1*1*sin(\t r)}); 
			\draw [shift={(2,2)},line width=1pt]  plot[domain=-1.57:1.57,variable=\t]({1*0.9*cos(\t r)+0*0.9*sin(\t r)},{0*0.9*cos(\t r)+1*0.9*sin(\t r)}); 

			\draw (1.7,2.95) node[anchor=east] {$\lambda=1$};
			\draw [dash pattern=on .5pt off 1.5pt,line width=.5pt] (1.8,2.95)-- (1.7,2.95);
			\draw (1.7,2) node[anchor=east] {$\lambda=\frac12$};
			\draw [dash pattern=on .5pt off 1.5pt,line width=.5pt] (1.8,2)-- (1.7,2);
			\draw (1.7,1.05) node[anchor=east] {$\lambda=0$};
			\draw [dash pattern=on .5pt off 1.5pt,line width=.5pt] (1.8,1.05)-- (1.7,1.05);
			\draw (2.62,2.7) node[anchor=south west] {$|\zeta|=1$};
			\draw (1.97,2.99) node[anchor=south] {$\psi_+$};
			\draw (1.97,1.01) node[anchor=north] {$\psi_-$};
		\end{tikzpicture}\\
		\emph{Baby maser, \ie,\\ $\alpha_n:=0$, $\beta_n:=1$.}\end{small}
	\end{minipage}
	\caption[Mixing]{These figures show the parameter regions\footnotemark[\value{footnote}]{} for which $T_\psi$ with the specified choices of model parameters is \emph{weakly mixing}. Blue areas indicate that $T_\psi$ is weakly mixing, red areas indicate that $T_\psi$ is not weakly mixing. For areas with blue and red lines either case can be true, and regions for which we do not have any results are left blank.
	}
\end{figure}
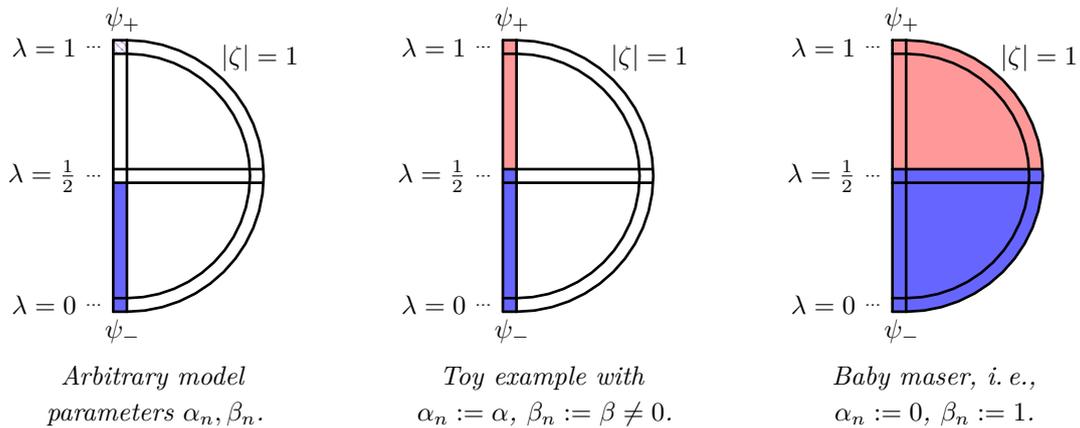\vspace{2ex}

\subsection{Connection Between Ergodicity and Irreducibility}

It is well known (and easy to see) that a ucp\nobreakdash-map that admits a faithful family of invariant normal states is ergodic if and only if it is irreducible. 
If a ucp\nobreakdash-map does not admit a faithful family of invariant normal states neither of the two implications hold in general.

Indeed, let $T_{\psi_-}$ be the transition operator corresponding to the state $\psi_-$ on~$M_2$ ($\lambda=0$, $\zeta=0$). As we have seen in Example~\vref{expl:psi+-}, this ucp\nobreakdash-map is ergodic but not irreducible for any choice of model parameters $\alpha_n,\,\beta_n$.

On the other hand, let $\psi$ be a state on $M_2$ parametrized as in Equation~\vref{eq:density} with $\nicefrac12<\lambda< 1$ and $T_\psi$ the corresponding transition operator of the baby maser ($\alpha_n=0$, $\beta_n=1$). Then $T_\psi$ is irreducible by Theorem~\vref{thm:irred_weaklymix} but the fixed space of $T_\psi$ is infinite dimensional (\cf Theorem~\vref{thm:babymaser_mixing}), \ie, $T_\psi$ is not ergodic.

\appendix
\section{Appendix}	\label{appendix}
In this section we include some laborious calculations for our transition operators $T_\psi$ as defined in Section \ref{sec:bndchain}. Throughout this section let $\psi$ be a  state on $M_2$ parametrized as in Equation \vref{eq:density} with $0 \le \lambda \le 1$ and $\zeta \in \D$ and we set $\nu:=i\zeta\sqrt{\lambda(1-\lambda)}$.

By Proposition \vref{prop:KrausDecomp} a Kraus decomposition of the transition operator $T_\psi$ is given by
	\begin{gather*}
		T_\psi(x) = \lambda (t_1^* x t_1 + t_2^* x t_2) + (1-\lambda)(1-\abs{\zeta}^2) (t_3^* x t_3 + t_4^* x t_4)
		\shortintertext{with}
		\begin{aligned}
			t_1 &:= s^* as + i \zeta \sqrt{\tfrac{1-\lambda}{\lambda}} \, s^* b \;,
			&
			t_2 &:= bs - i \zeta \sqrt{\tfrac{1-\lambda}{\lambda}} \, a \;,
			\\
			t_3 &:= s^* b \;, 
			& 
			t_4 &:= a \;.
		\end{aligned}
	\end{gather*}
The operators $t_1,\ldots,t_4$ and their adjoints act on the canonical orthonormal basis $(e_n)_{n \in \N}$ of $\H = \ell^2(\N)$ as follows
\begin{small}
\begin{align*}
	t_1\,e_n &= \begin{cases} \alpha_{n+1}\, e_n + i\zeta\sqrt{\tfrac{1-\lambda}{\lambda}}\,\beta_{n}\, e_{n-1} & \mbox{if }n\geq 1\,,\\
								\alpha_{n+1}\, e_n & \mbox{if }n=0\,, \end{cases} \\
	t_2\,e_n &= \beta_{n+1}\, e_{n+1} - i\zeta\sqrt{\tfrac{1-\lambda}{\lambda}}\,\alpha_{n}\, e_{n}\,,\\
	t_3\,e_n &=  \begin{cases} \beta_{n}\, e_{n-1} & \mbox{if }n\geq 1\,,\\
								0 & \mbox{if }n=0\,, \end{cases} \\
	t_4\,e_n &=\alpha_{n}\, e_{n}\,,
	\shortintertext{and}
	t_1^*\,e_n &= \alpha_{n+1}\, e_n - i\bar\zeta\sqrt{\tfrac{1-\lambda}{\lambda}}\,\beta_{n+1}\, e_{n+1}\,,\\
	t_2^*\,e_n &= \begin{cases}\beta_{n}\, e_{n-1} + i\bar\zeta\sqrt{\tfrac{1-\lambda}{\lambda}}\,\alpha_{n}\, e_{n} & \mbox{if }n\geq 1\,,\\
	i\bar\zeta\sqrt{\tfrac{1-\lambda}{\lambda}}\,\alpha_{0}\, e_{0} & \mbox{if }n=0\,, \end{cases} \\
	t_3^*\,e_n &= \beta_{n+1}\, e_{n+1}\,,\\
	t_4^*\,e_n &=\alpha_{n}\, e_{n}\,.
\end{align*}
\end{small}

Using the Bra-ket notation $\ket{e_n}\bra{e_m}:=e_{n,m}$ we see that this yields
\begin{footnotesize}
\begin{align*}
	t_1^*\,e_{n,m}\,t_1 =\; & t_1^*\;\ket{e_n}\bra{e_m}\;t_1 = \ket{t_1^*\,e_n}\bra{t_1^*\,e_m}
	\\
	=\;& \alpha_{n+1}\alpha_{m+1}\,e_{n,m}+\tfrac{1-\lambda}{\lambda}\abs{\zeta}^2\,\beta_{n+1}\beta_{m+1}\,e_{n+1,m+1}\;
	\\
	&-i\bar\zeta\sqrt{\tfrac{1-\lambda}{\lambda}}\,\beta_{n+1}\alpha_{m+1}\, e_{n+1,m}+i\zeta\sqrt{\tfrac{1-\lambda}{\lambda}}\,\alpha_{n+1}\beta_{m+1}\, e_{n,m+1}\,,
	\\
	t_2^*\,e_{n,m}\,t_2 =\;& \tfrac{1-\lambda}{\lambda}\abs{\zeta}^2\,\alpha_{n}\alpha_{m}\,e_{n,m} +
	\\
	&+\begin{cases}
		\beta_{n}\beta_{m}\,e_{n-1,m-1}-i\zeta\sqrt{\tfrac{1-\lambda}{\lambda}}\,\beta_{n}\,\alpha_{m}\, e_{n-1,m}+i\bar\zeta\sqrt{\tfrac{1-\lambda}{\lambda}}\,\alpha_{n}\beta_{m}\, e_{n,m-1} &\mbox{if }n, m\geq1\,,\\
		i\bar\zeta\sqrt{\tfrac{1-\lambda}{\lambda}}\,\alpha_{0}\beta_{m}\, e_{0,m-1} &\mbox{if }n=0,\; m\geq1\,,\\
		-i\zeta\sqrt{\tfrac{1-\lambda}{\lambda}}\,\beta_{n}\,\alpha_{0}\, e_{n-1,0} &\mbox{if }n\geq1,\; m=0\,,\\
		0 &\mbox{if }n=m=0\,,
	\end{cases}\\
	t_3^*\,e_{n,m}\,t_3 =\;& \beta_{n+1}\beta_{m+1}\,e_{n+1,m+1}\,,\\
	t_4^*\,e_{n,m}\,t_4 =\; & \alpha_{n}\alpha_{m}\,e_{n,m}\,.
\end{align*}
\end{footnotesize}

Hence for $m,n\in\N$ we have
\begin{small}
\begin{align*}
	T_\psi(e_{n,m})\;=\ &\bigl(\lambda\alpha_{n+1}\alpha_{m+1}+(1-\lambda)\alpha_{n}\alpha_{m}\bigr)e_{n,m} 
						+\bigl((1-\lambda)\beta_{n+1}\beta_{m+1}\bigr)e_{n+1,m+1}\\
					&+\bigl(\alpha_{m+1}\beta_{n+1}\bar\nu\bigr)e_{n+1,m}
						+\bigl(\alpha_{n+1}\beta_{m+1}\nu\bigr)e_{n,m+1}\\
					&+\begin{cases}
						\bigl(\lambda\beta_{n}\beta_{m}\bigr)e_{n-1,m-1}  - \bigl(\alpha_{n}\beta_{m}\nu\bigr)e_{n-1,m}  - \bigl(\alpha_{n}\beta_{m}\bar\nu\bigr)e_{n,m-1} &\mbox{if }n, m\geq1\,,\\
						- \bigl(\alpha_{0}\beta_{m}\bar\nu\bigr)e_{0,m-1} &\mbox{if }n=0,\; m\geq1\,,\\
						- \bigl(\alpha_{n}\beta_{0}\nu\bigr)e_{n-1,0} &\mbox{if }n\geq1,\; m=0\,,\\
						0 &\mbox{if }n=m=0\,.
						\end{cases}
\end{align*}
\end{small}

\begin{prop}	\label{prop:neighboringTransitions}
	For the diagonal projections $p_{[m,n]}=\sum\limits_{k=m}^n p_k=\sum\limits_{k=m}^n e_{k,k}$ ($m\leq n$) we have the following relations
	\begin{equation*}
		p_{[m+1,n-1]} \le T_\psi(p_{[m,n]}) \le p_{[m-1, n+1]} \;.
	\end{equation*}
\end{prop}
\begin{proof}
	The second inequality $T_\psi(p_{[m,n]}) \le p_{[m-1, n+1]}$ is obvious by the formula for $T_\psi(e_{n,m})$ above. For the first inequality consider a normal state $\varphi$ on $\BH$ with $\supp\varphi\leq p_{[m+1, n-1]}$. Then $\varphi$ is a (finite) convex combination of vector states $\scal{\;\cdot\ \xi_i\,,\,\xi_i}$ with $\xi_i\in p_{[m+1, n-1]}\,\H$. Let $t_1,\ldots,t_4$ be the operators of the Kraus decomposition of Proposition \vref{prop:KrausDecomp}. Then  $t_j\,\xi_i\in p_{[n,m]}\,\H$, $1\leq j\leq 4$, (see the calculations above) and hence
	\begin{align*}
		\scal{T_\psi(p_{[m,n]})\,\xi_i\,,\,\xi_i} =\; & \lambda \bigl( \scal{t_1^*\,p_{[m,n]}\,t_1\, \xi_i,\xi_i} + \scal{t_2^*\,p_{[m,n]}\,t_2\, \xi_i,\xi_i} \bigr)
		\\
		& + (1-\lambda)(1-\abs{\zeta}^2) \bigl( \scal{t_3^*\,p_{[n,m]}\,t_3\, \xi_i,\xi_i} + \scal{t_4^*\,p_{[n,m]}\,t_4\, \xi_i,\xi_i} \bigr)
		\\
		=\; & \lambda \bigl( \scal{t_1^*\,t_1\, \xi_i,\xi_i} + \scal{t_2^*\,t_2\, \xi_i,\xi_i} \bigr)
		\\
		& + (1-\lambda)(1-\abs{\zeta}^2) \bigl( \scal{t_3^*\,t_3\, \xi_i,\xi_i} + \scal{t_4^*\,t_4\, \xi_i,\xi_i} \bigr)\\
		=\; & \scal{T_\psi(\one)\xi_i\,,\,\xi_i} = \scal{\xi_i\,,\,\xi_i} = 1.
	\end{align*}
	This implies that $\varphi(T_\psi(p_{[n,m]}))=1$. Hence $p_{[n+1,m-1]}\,T_\psi(p_{[n,m]})\,p_{[n+1,m-1]}=p_{[n+1,m-1]}$, which yields the conclusion.
\end{proof}

\section*{Acknowledgments} Most of the results in this paper were derived during several mini-workshops of our research group at the Technische Universit\"at Darmstadt, at which the bachelor or master students Rebekka Burkholz, Jan D\"orner, Anja Kleinke, Florian Steinberg, and Stefan Wiedenmann were also participating.

\bibliographystyle{amsalpha_ag}
\phantomsection
\addcontentsline{toc}{section}{References}
\begin{small}
\bibliography{bndchains}
\end{small}
\label{LastPage}

\end{document}